\documentclass[12pt]{amsart}
\usepackage{amsmath, amscd, color,mathrsfs,graphicx}
 
\theoremstyle{plain}
\newtheorem{thm}{Theorem}
\newtheorem{cor}{Corollary}

\newtheorem{lem}{Lemma}
\newtheorem{prop}{Proposition}

\theoremstyle{definition}
\newtheorem{defin}{Definition}
\theoremstyle{remark}
\newtheorem{rem}{Remark}
\theoremstyle{remark}
\newtheorem{notat}{Notation}

\newcommand{\Z}{\mathbb{Z}}

\newcommand{\R}{\mathbb{R}}

\def\slr{\text{SL}(2,\R)}

\def\pslr{\text{PSL}(2,\R)}

\usepackage{hyperref}

\begin{document}
 

\title[Commensurable continued fractions]{Commensurable continued fractions}
\author{ Pierre Arnoux}
\address{Institut de Math\'ematiques de Luminy (UMR6206 CNRS),
         163 Avenue de Luminy, case 907,
         13288 Marseille cedex 09,
         France}
\email{arnoux\@ iml.univ-mrs.fr}
\author{Thomas A. Schmidt}
\address{Oregon State University\\ Corvallis, OR 97331}
\email{toms@math.orst.edu}

\date{03 September 2013}
\thanks{The second-named author thanks in particular the Universit\'e d'Aix-Marseille for friendly hospitality during the completion of this work.}

 
\begin{abstract}  We compare two families of continued fractions algorithms, the symmetrized Rosen algorithm and the Veech algorithm.   Each of these algorithms expands real numbers in terms of certain algebraic integers. We give explicit models of the natural extension of the maps associated with these algorithms; prove that these natural extensions are in fact conjugate to the first return map of the geodesic flow on a related surface;  and, deduce that, up to a conjugacy, almost every real number has an infinite number of common approximants for both algorithms.
\end{abstract}

\maketitle

\tableofcontents

\section{Introduction} 
{\em How does one compare two continued fraction algorithms related to commensurable Fuchsian groups?}  This question naturally arose as we considered the continued fractions that Arnoux  and Hubert \cite{AH} introduced and called {\em Veech continued fractions}.   
For each $q=2n\ge 8$ this is a continued fraction algorithm expressing the (inverse of) slopes of lines through the origin in $\mathbb  R^2$ in terms of a geometric expansion related to the regular $q$-gon.    The algorithm was inspired by work of Veech on flat surfaces, and related to this is the fact that these continued fractions are naturally expressed in terms of elements of a Fuchsian triangle group of signature $(n, \infty, \infty)$.    But,   each Fuchsian group of this signature is conjugate to 
a subgroup of index two of the Hecke group $H_q$ for which Rosen some 50 years earlier gave related continued fractions.    The natural question here is:   Did they rediscover an old algorithm?  And, if not, can one measure in some reasonable sense how closely related the two algorithms are for each fixed index $q$? 
 
Here we make minor adjustments to the two algorithms.  First, we use a symmetric form of the Rosen algorithm, multiplying by $-1$ for $x<0$.  This slight change, which leaves unchanged the approximations produced for any real $x$,  gives an algorithm that is expressed directly in terms of the fractional linear action of the corresponding Fuchsian group.   Second,  we modify the Veech continued fraction so as to expand slopes of lines.   

\bigskip
\begin{rem}
Note that H. ~Nakada \cite{N} introduced the above variant of the Rosen fractions, and indeed Mayer and Stromberg \cite{MS}   refer to these as Hurwitz-Nakada continued fractions.    We refer to them as the  {\em symmetrized Rosen continued fractions}.

The authors first heard A.~ Haas suggest using variants of the Rosen fractions in the study of natural extensions, in particular using so-called backwards fractions, see his later joint work with Gr\"ochenig \cite{GH}.     
See \cite{DKS} for an infinite family of variants for each index $q\,$.  See also \cite{SS} for an application of variants of the Rosen fractions to the study of length spectra of the hyperbolic surfaces uniformized by the Hecke groups.
\end{rem}
\bigskip

Each of the algorithms is given by a piecewise M\"obius (equivalently, increasing fractional linear) interval map and thus locally leaves invariant the equivalence class of measures that are absolutely continuous with respect to Lebesgue measure.  Our approach is to find  for each of these interval maps a dynamical system defined on a closed domain of $\mathbb R^2$ that fibers over the interval of definition and for which there is an obvious invariant measure that is absolutely continuous with respect to Lebesgue measure there.    The marginal measure,  given by integrating along fibers,  is then invariant for the interval map.    One can then reasonably hope that the two dimensional system is (a model of) the natural extension of the one dimensional system.     In our setting, as in \cite{AS}, the various   M\"obius transformations associated to each interval map generate a Fuchsian group,  and the two dimensional system is in fact isomorphic to a map given by returns to a cross-section for the geodesic flow on the unit tangent bundle of the appropriately corresponding hyperbolic surface. 
 It then follows that this cross-section is a natural extension of the interval map,  and also that  the interval map is ergodic.     Composing the Rosen map with itself and appropriately conjugating,    allows a direct comparison for each index of the natural extensions of the two interval maps, in the form of cross-sections for the geodesic flow on the unit tangent space of the same surface. 

To be precise, we give abbreviated forms of the definitions of the maps.   Full definitions and further discussion are given below. 
\begin{notat}   
Let  $q = 2 n$ be an even integer greater than 4.   
Let $U$ denote the counter-clockwise rotation by $\pi/q$, and let $R=U^2$.   
Let  $\mu = \mu_q= 2 \cot \pi/q$ and $I = [\, - \mu/2, \mu/2\,]$.   Let $\text{sign}(x) = \pm 1$ be the usual signum function.  

For a real matrix $M$,  we denote by  $M\cdot x$  the M\"obius action of the matrix $M$ on an (extended) real $x$; that is,    $M\cdot x =   \frac{ax+b}{cx+d}$, where  $M=  \begin{pmatrix} 
a&b\\c&d\end{pmatrix}$.

\end{notat}

\begin{defin}   [Conjugate Rosen map]
Let $k : I\to I$ be the map  that sends $x \in I$ to its image under the rotation $U^{-\text{sign}(x)}$ followed by the unique translation by an integer multiple of $\mu$ such that this image lies in $I$.  The conjugate Rosen map  $r : I\to I$  is the square of $k$.
 \end{defin}

 \begin{defin} [Veech map]
 Let $a : I\to I$ be the map that sends $x\in I$ to its image under the unique power of $R^{-1}$ that maps $x$ outside of $I$ followed by the 
the unique translation by an integer multiple of $\mu$ bringing this image back into $I$.     
The Veech map  $v:I\to I$  be the parabolic acceleration of $a$ at the end points of $I$; that is,  for $x$ such that $a(x) = -\mu + R\cdot x$, let $v(x) = a^k(x)$ with $k$ minimal such that $a^{k+1}(x) \neq   -\mu + R\cdot a^k(x)$,  and similarly for $x$ such that $a(x) = \mu - R\cdot x$.   
\end{defin} 

\begin{rem}\label{r:theMaps}  
We show in Section~\ref{s:conjugRosen} that the map $r(x)$ is conjugate by way of an explicit element of $\slr$ to the square
of the symmetrized Rosen map.   The map $v(x)$ is the multiplicative Veech map of \cite{AH} except that it expresses an action on slopes as opposed to inverse slopes.  
\end{rem}

We find that the continued fraction map of \cite{AH} is certainly not simply a variant of the Rosen map, but still that the two are truly comparable.   The main result of the paper is the following theorem.

\begin{thm}\label{t:comp}  Suppose that $q = 2 n$ is an even integer greater than four.  
Each of the maps $r$ and $v$  is a factor of the first return map to a respective cross-section of the geodesic flow on the unit tangent bundle of the hyperbolic surface uniformized by a Fuchsian triangle group of signature $(q, \infty, \infty)$.       Furthermore,  these two cross-sections intersect in a region of finite measure.   In particular,  for almost every $x\in I$, the  $r$- and $v$-orbits  
agree infinitely often, and thus the respective sequences of approximants agree infinitely often. 
\end{thm}

In fact,  the area of intersection of the two cross-sections  limits when $q$ tends to infinity
 to be $1/3$  of the area of the cross-section factoring over $v(x)$, but is asymptotically of zero relative size in that over $r(x)$. See Section~\ref{s:comparison}.    

\subsection{Natural extensions and geodesic flows}\label{ss:NatExtGeoFloIntro}
We consider a measure-preserving dynamical system $(T, I, \mathscr B, \nu)$, where $T: I \to I$ is  a measurable map on the measurable space $(I,\mathscr B)$ which preserves the measure $\nu$ ($\nu$ is usually a probability measure).

A natural extension of  the dynamical system  $(T, I, \mathscr B, \nu)$  is an invertible system  $(\mathcal T, \Omega,\mathscr B', \mu)$  with a surjective projection $\pi: \Omega \to I$ making   $(T, I,  \mathscr B, \nu)$ a factor of $(\mathcal T, \Omega, \mathscr B', \mu)$,  and such that any other invertible system with this property has its projection factoring through $(\mathcal T, \Omega, \mathscr B', \mu)$.   The natural extension of a dynamical system exists always, and is unique up to measurable 
isomorphism, see \cite{Roh}.     Informally, the natural extension is given by appropriately giving to (forward) $T$-orbits an infinite (in general) past; an abstract model of the natural extension is easily built using inverse limits.
 
There is an efficient heuristic  method for explicitly determining a geometric model of  the natural extension of an interval map when this map is given (piecewise) by M\"obius transformations.    If the map  is given by  generators of a Fuchsian group of finite covolume, then one can hope to realize the natural extension as a factor of a section of the geodesic flow on the unit tangent bundle of the hyperbolic surface uniformized by the group.    To do this explicitly, we expand the method explained and applied in \cite{AS},  derived from \cite{Ar}, as we briefly summarize here.    
 
    Using the M\"obius action of $\text{SL}_2(\mathbb R)$ on the Poincar\'e upper-half plane $\mathbb H$,  by identifying a matrix with the image of $z=i$ under it,   we can identify $\text{SL}_2(\mathbb R)/ \text{SO}_2(\mathbb R)$ with $\mathbb H$.   Similarly,   $\text{PSL}_2(\mathbb R) = \text{SL}_2(\mathbb R)/\pm I$ can be identified with the  unit tangent bundle of $\mathbb H$.   
The geodesic flow acts on a surface's unit tangent bundle:  Given a time $t$ and a unit tangent vector $v$,   
 follow the unique geodesic to which $v$ is tangent for arclength $t$ in the direction of $v$. The image,  $g_t(v)$, under the geodesic flow  is  the unit vector tangent  to the geodesic at the end point of the geodesic arc.    The hyperbolic metric on $\mathbb H$ corresponds to an element of arclength satisfying $ds^2 = (dx^2 + dy^2)/y^2$ with coordinates $z = x + i y$.   In particular,  for $t>0$,   the points $z = i$ and $w = e^t i$ are at distance $t$ apart.  Since $\text{SL}_2(\mathbb R)$ acts by  isometries on $\mathbb H$,   the geodesic flow on its unit tangent bundle is given by sending $A\in \text{PSL}_2(\mathbb R)$ to $A g_t$,   where $g_t = \begin{pmatrix} e^{t/2}&0\\0&e^{-t/2}\end{pmatrix}$.    

(In all that follows, we commit the standard abuse of notation of representing a class in 
$\pslr$ by a matrix of $\slr$ in this class.)   The following subset provides a transversal to the geodesic flow on $\mathbb H$, and is central to all that follows.

\begin{equation}\label{eq:thatsA} 
{\mathcal  A} = \left\{ \begin{pmatrix} x&xy-1\\1&y\end{pmatrix} \,\right\}
        \subset \text{PSL}(2,{\mathbb  R}).
\end{equation}

 Now fix a Fuchsian group $\Gamma$ of finite covolume. 
 Given  $A \in {\mathcal  A}$ and $M = \begin{pmatrix} 
a&b\\c&d\end{pmatrix} \in \Gamma$, suppose that $cx+d >0$ and  let $t_0 = -2 \log (c x + d)$.  Then 
\[ MA g_{t_0} = \begin{pmatrix} M\cdot x&*\\1&(cx+d)^2y- c(cx+d)\end{pmatrix}\]
Here as defined above,  $M\cdot x =   \frac{ax+b}{cx+d}$.

Focusing on the diagonal entries of the above matrix,  we find the following transformation.  
\begin{equation}\label{eq:2Daction}
{\mathcal  T}_M: (x,y) \to (\,M \cdot x,(cx+d)^2y- c(cx+d)\,)\,.
\end{equation}
An elementary calculation shows that the Jacobian matrix of ${\mathcal  T}_M$ has determinant one. 
Thus,  ${\mathcal  T}_M$ is Lebesque measure preserving on ${\mathbb  R}^2$.

\begin{defin}\label{d:pwMoebiusMap}   An interval map $f: I \to I$ is called a  {\em piecewise M\"obius map}  if there is a 
  partition of $I$ into intervals 
  $I=\bigcup \, I_\alpha$ and a set 
${\mathcal M}:=\{M_\alpha\}$ of elements  $\pslr$ such that the restriction of $f$ to $I_{\alpha}$ is exactly given by 
 $ I_\alpha \ni x \mapsto
M_\alpha  \cdot x\,$.     We call the subgroup of $\pslr$ generated by ${\mathcal M}$ the {\em group generated by $f$}, and denote if by $\Gamma_f\,$.   
\end{defin} 

 Each of the interval maps that we consider is piecewise M\"obius and locally expanding almost everywhere.  
 Given such an interval map $f: I\to I$ for some interval $I$, we consider the transformation $\mathcal T$,  defined as ${\mathcal  T}_M$ on fibers above the subinterval  where $f(x) = M\cdot x$.    Upon identifying a
 set of  positive measure $\Omega \subset \mathbb R^2$ fibering over the interval such that  $\mathcal T$ is bijective (up to measure zero) on $\Omega$, using that  the  ${\mathcal  T}_{M}^{-1}$ are expansive on $y$-values, one could  confirm that one has found a model of the natural extension of $(f, I, \mathscr B, \nu)$, where $\nu$ is the marginal measure given by integrating along fibers, and $\mathscr B$ denotes the standard Borel $\sigma$-algebra. (See, say,  the proof of Theorem~1 of \cite{KSS}, on p. 2219 there.)

An alternative proof uses the Anosov property of the geodesic flow. Since the map $\mathcal T_M$ is defined by way of the 
transversal  $\mathcal A$, the existence of an invariant compact set of positive measure $\Omega$ as above implies that $(f, I, \mathscr B, \nu)$ is a factor of a map $\Phi$  given by returns to a subset $\Sigma$ (included in the projection of  $\mathcal A$) of the unit tangent bundle of $\Gamma_f\backslash \mathbb H$.   The set $\mathcal A$ is fibered over the real $x$-line by elements corresponding to horocycles.   Since the geodesic flow is exponentially contracting on these horocycles,  the inverse of $\Phi$ is expansive on the ``$y$''-coordinate. Suppose that we have two bi-infinite orbits $(x_i,y_i)$ and $(x_i,y'_i)$ which project to the same orbit of  $T$ (i.e., for all $i\in \Z$, $x_{i+1}=T(x_i)$). Then $(x_i,y_i)$ and $(x_i,y'_i)$ belong to the same horocycle, and must be equal by the expansiveness of the inverse of $\Phi$. Hence there is exactly one two sided $\Phi$-orbit which projects to a given bi-infinite orbit of $T$, and thus we have indeed found the natural extension.   Finally, if this $\Phi$ is actually given by {\em first} returns to $\Sigma$,  we have that $f$ is a factor of the cross-section given by $\Sigma$.  

\subsection{Ergodicity to discover the domain of the natural extension}

One of the goals of the present work is to give non-trivial examples illustrating the method sketched just above.  
In practice, in order to  discover an explicit domain $\Omega$ on which $\mathcal T$ is bijective,  one uses the ergodicity of 
this map on the purported domain of definition.     Informally stated, taking (a sufficiently large initial portion of) the ${\mathcal  T}$-orbit of a sufficiently general point in $\Omega$ will trace out (a sufficiently large portion of) $\Omega$.   (And, in fact,  $\Omega$ is appropriately attractive for ${\mathcal  T}$ (see our \cite{AS3}), so that one can begin with virtually any point of the plane.)   Moreover, since Lebesgue measure is invariant, almost any orbit will cover the set $\Omega$ uniformly.

   The interval maps that we consider in this work are defined in an algorithmic fashion and are of an arithmetic nature.  It is easily verified that 
for any of these, any real transcendental $x$ in its interval of definition must have an infinite, non-periodic orbit.  Thus,  our Figures ~\ref{extNatRosSymFig}, \ref{natExtVeechAddFig}, \ref{natExtVeechMultFig}, \ref{natExtOverlayFig} each show  an initial portion of the orbit of $(\pi/10, 0)$ for the respective transformations.     Of importance is now the fact that  for non-zero real $\delta$, 
\begin{equation}\label{eq:TMonDelta}
\mathcal T_M:   (x,1/(x - \delta)\,) \mapsto (M\cdot x, 1/(M\cdot x - M\cdot \delta)\,)\,,
\end{equation}
as one verifies by considering Equation \eqref{eq:2Daction} while writing the relevant matrices in the form $\begin{pmatrix} x&y \delta\\1&y\end{pmatrix}$.   The regions so sketched are sufficiently clear to guess equations of the form $y = 1/(x - \delta)$ for the boundary pieces in these examples.    Generalization of these formulas to the families of interval maps considered is an easy matter.   One can verify that $\mathcal T$ is bijective on a candidate $\Omega$ by careful use of the definition of the interval map at hand.  The (Lebesgue measure) area of this region gives the normalizing constant so as to give the probability measure on $\Omega$ with respect to which $\mathcal T$ is ergodic; integrating along vertical fibers gives the invariant measure for the interval map.

\subsection{Comparing algorithms}

With the natural extensions in hand,  we can move towards comparing ``commensurable'' interval maps.  There are two problems here; while having distinct conjugate group is not a major problem (conjugate Fuchsian groups uniformize the same 
hyperbolic surface, and one just need to use a suitable change of coordinates for one of the maps), interval maps associated with strict subgroups can create difficulties. But in this case, the second problem is easy to solve: the Veech map is associated to a group conjugate to a subgroup of index two of the Hecke group $H_q$, and each matrix associated with the Rosen map is contained in the complement of this subgroup of index two; hence the square of the Rosen map is defined, up to a conjugacy, in terms of the same group as the Veech map.
Hence we find that we can actually intersect the natural extensions (which are given as cross-sections of the geodesic flow on the unit tangent bundle of the shared hyperbolic surface).   The relative size of this intersection to the original sections gives a reasonable measure of how related the interval maps are.   

In particular,  we find that the
Veech continued fractions are new, as (for each fixed $n$) the intersection of the natural extension with that of the corresponding Rosen-type continued fractions is a proper subset of each (of course, we describe much more detail than this).   Still, because of the ergodicity of the maps and the fact that this intersection has positive measure,   for almost every $x$, the sequences of approximants defined by the two continued fraction algorithms agree infinitely often.

\subsection{Outline}     In the next section we discuss the symmetrized Rosen fractions, and in particular describe the process of solving for the planar model of the natural extension.  In Section~\ref{s:VeechFracs}, we treat the 
Veech continued fractions, beginning with an additive version, with infinite invariant measure, and proceeding to a multiplicative version, by explicitly accelerating (that is,  repeatedly composing the map with itself) in the vicinity of the two parabolic fixed points.   We give the domain of the planar model.  Section~\ref{s:conjugRosen} is devoted to doubling and conjugating the Rosen maps so that they are defined over the same Fuchsian group as the Veech maps.      Here also, explicit domains of planar models are determined.   In \ref{s:firstReturn},  we turn to the idea of first return to cross-sections of the geodesic flow on the unit tangent bundle of the hyperbolic surface uniformized by the Fuchsian group at hand.   In particular we give a more direct proof than in our \cite{AS2} of the fact that the product of the entropy times the area of the planar natural extension is greater than or equal to the volume of this unit tangent bundle and that equality holds if and only if the natural extension map is induced by the first return under the geodesic flow to the cross-section.   In Subsection~\ref{ss:oursFirst}, we show that for each index, both of our maps are of first return type.  Section~\ref{s:comparison} then gives the comparison;  the two natural extension domains meet in a region of positive measure,  with relative area limiting to $1/3$ of that for the Veech maps, and to zero in the Rosen maps.   Furthermore,  for $q$ divisible by four,  we give explicit regions showing that the unboundedness of the 
number of returns to its natural extension of the Veech map before returning to the intersection.

\subsection{Thanks}   The present work remained in nascent form for several years.   Indeed,  results of \cite{AS}, long a subsection and appendix of early versions,  are referred to in \cite{HS}.   Similarly,   portions of \cite{AS2} formed a part of early versions.    We  thank P. Hubert for prompting and support.

\section{The symmetrized Rosen algorithm}   

\subsection{The Rosen algorithm and its symmetrization} 
The Rosen algorithm \cite{R} aims to give a representation in terms of continued fractions with partial quotients rational integral multiples of the fixed number $\lambda$ in the form 

\[ x = \cfrac{\varepsilon_1}{b_1 \lambda +\cfrac{\varepsilon_2}{b_2 \lambda + \cfrac{\varepsilon_3}
                                                                                                                                                   {\ddots}}}\;,
\]
where each $\varepsilon_i$ is $1$ or $-1$ and the $b_i$ are natural numbers. This makes sense when $0<\lambda<2$, and is particularly interesting when  $\lambda = 2 \cos \pi/n\,$ with $n\ge 3$, because in that case  it  satisfies a Markov property, and it is related to the Hecke group $H_n$ (for $n=3$ we recover the nearest integer continued fraction).   The case of interest for us is where $\lambda = \lambda_q = 2 \cos \pi/q\,$, with $q = 2n$ an even  natural number, of value at least $4$.  

Setting $ x = \varepsilon_1/( \, b_1 \lambda + f(x)\,)\,$, we see that $\varepsilon_1= \text{sign}(x)$,  $b_1=b(x)= \bigg\lfloor \, \dfrac{1}{\vert\, x\, \vert\, \lambda} + \dfrac{1}{2}\,\bigg\rfloor$, and we find an interval map $f$ on $J = J_q = [-\lambda/2,  \lambda/2]\,$:

\[
\begin{aligned} 
f: J &\to J\\
x &\mapsto \dfrac{1}{\vert\,x\,\vert} - \bigg\lfloor \, \dfrac{1}{\vert\, x\, \vert\, \lambda} + \dfrac{1}{2}\,\bigg\rfloor\, \lambda\;.
\end{aligned}
\]

One computes the Rosen expansion of $x$ as $b_n=b(f^{n-1}(x))$ and $\varepsilon_n=\text{sign}(f^{n-1}(x))$.

The absolute value in Rosen's algorithm introduces an asymmetry that unnecessarily complicates the study of $f\,$, since for positive $x$ it invokes fractional linear maps of the form $1/x - b \lambda\,$, given by matrices of determinant $-1\,$.    We thus discuss an algorithm that is almost that of Rosen, in that it determines the same set of partial quotients,  but which enjoys more symmetry.   This {\em symmetric Rosen map} is the following: 
\begin{equation}\label{e:rosenSymDef} 
\begin{aligned}
h: J &\to J\\
x &\mapsto \dfrac{-1}{x} - \bigg\lfloor \, \dfrac{-1}{x\, \lambda} + \dfrac{1}{2}\,\bigg\rfloor\, \lambda\;.
\end{aligned}
\end{equation}
The corresponding continued fraction gives a representation of $x$ in the form:

\[ x = \cfrac{-1}{a_1 \lambda -\cfrac{1}{a_2 \lambda - \cfrac{1} {\ddots}}}\;,
\]
where the $a_i$ are nonzero (positive or negative) integers. The geometry of this map are straightforward:   one sends $x$ to  $-1/x$, and then translates by an integral multiple  of $\lambda$ to bring the result into the interval $J\,$.   This translation is uniquely defined except for when the image lies at an endpoint of $J\,$, this set of measure zero we systematically ignore except as otherwise stated.   

Setting $a(x)=\bigg\lfloor \, \dfrac{-1}{ x\,  \lambda} + \dfrac{1}{2}\,\bigg\rfloor$, one checks that $a_n=a(h^{n-1}(x))$, and also that $a_n=(-1)^n \varepsilon_1\ldots \varepsilon_n b_n$, so the two maps give the same expansion.

\begin{figure}
\begin{center}
\scalebox{0.8}{
\includegraphics{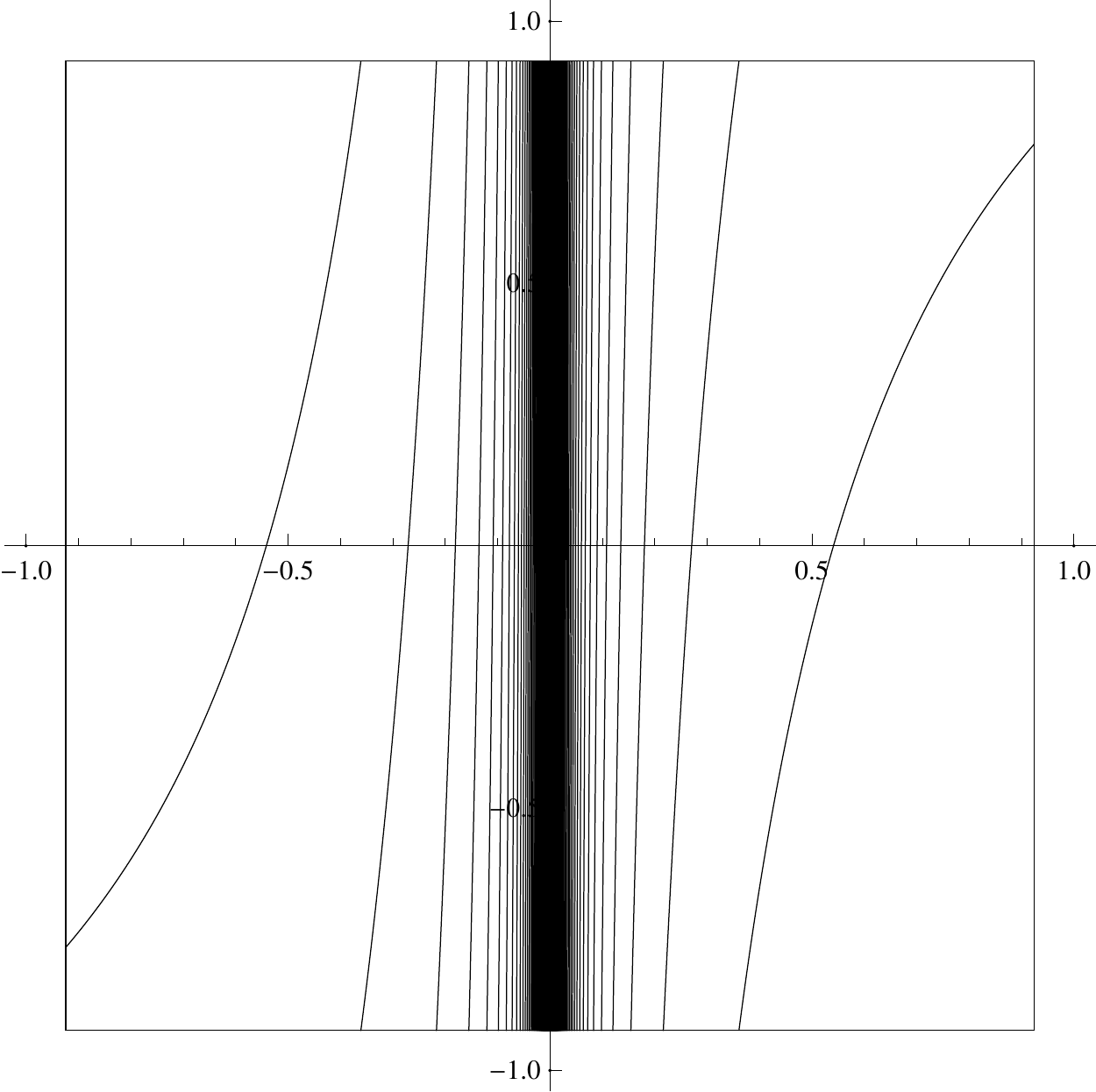}}
\caption{Representation of the graph of the symmetrized Rosen function, $q=8$}
\label{rosenSymFig}
\end{center}
\end{figure}

\subsection{Dynamics of the Rosen map}

The two functions $f$ and $h$ agree for $x <0\,$; note that other than when $\frac{-1}{x \lambda} + \frac{1}{2}$ is integral,  $h(-x) = - h(x)$.   The symmetrized function has the  advantage of being given, on each subinterval of continuity, by the standard action of elements of a Fuchsian subgroup of $\text{PSL}(2, \mathbb R)$,  the Hecke triangle group of index $q$, denoted here $H_q\,$: 

\[ H_q = \langle \, S, \, T\,  \rangle\,, \text{where}\;   S = S_q = \begin{pmatrix}1 &\lambda\\0&1\end{pmatrix},\,  T = \begin{pmatrix}0 &-1\\1&0\end{pmatrix}\;.
\]
These two generators satisfy the relations   $T^2 = \text{Id}$ and $(ST)^q = \text{Id}$ in $\text{PSL}(, 2, \mathbb R)\,$.   Note that $S \cdot x = x + \lambda$ and $T \cdot x = -1/x\,$.  Thus if $j =a(x)$, then $h(x) = S^{-j} T \cdot x\,$.  

The {\em cylinders} of the map $h$ are the sets $\Delta_j = \{x \in J \,|\, a(x)=j\}$. On $\Delta_j$ one has $h(x) = S^{-j} T\cdot x$. If  $j >1$ (resp. $j<-1$), then $\Delta_j=[\frac{-2}{(2j-1)\lambda}, \frac{-2}{(2j+1)\lambda})$ (resp. $[\frac{-2}{(2j-1)\lambda}, \frac{-2}{(2j+1)\lambda})$ ); each of these cylinders is {\em full}, in the sense that $h$ maps it onto the full interval $J$, as one can see on Fig. \ref{rosenSymFig}.   

The dynamics of the cylinders $\Delta_{\pm 1}$ are more complicated, and we  recall some observations of \cite{BKS}, where the original Rosen algorithm is addressed; of course, this and the symmetrized algorithm treat negative $x$ in the same way. As in \cite{BKS},  the orbit of $-\lambda/2$ is the key to the dynamics of the map; setting $\phi_j = f^j(-\lambda/2)$, an easy computation shows that $\phi_j=\frac{-\cos(j+1)\pi/q}{\cos j \pi/q}$ for $1\le j<n$, and in particular $\phi_{n-1} = 0\,$ (recall that $q = 2n\,$), after which the orbit is no longer defined.  This finite  sequence satisfies $\phi_0 = -\lambda/2< \phi1< \dots< \phi_{n-2} = -1/\lambda<-2/3\lambda<\phi_{n-1}=0$.  Hence the cylinder $\Delta_1$ is partitioned by the intervals $[\phi_j,\phi_{j+1})$, for $0\le j<n-3$, and the interval $[\phi_{n-2}, 2/3\lambda)$, with $h$ sending each interval to the next, and the last interval to $[0,\lambda/2)$. 

 These values are visible on Figure ~\ref{extNatRosSymFig} --- they are the $x$-coordinates of the discontinuities along the upper boundary of $\mathcal E\,$.    

\subsection{Solving for the planar extension of the symmetrized Rosen fractions}

\begin{figure}
\begin{center}
\scalebox{.7}{\includegraphics{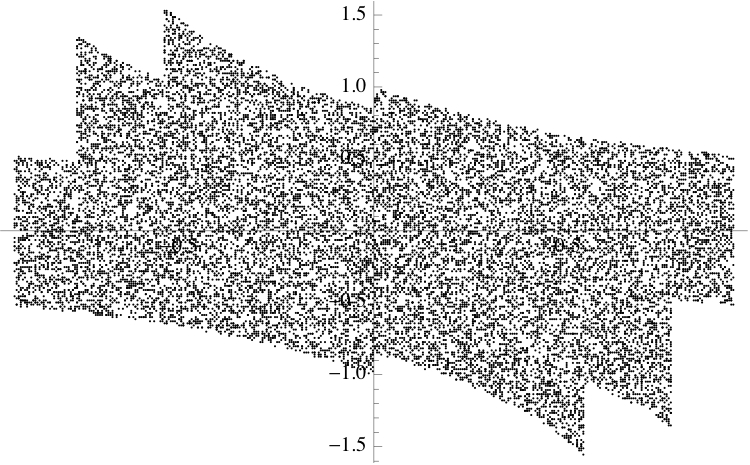}}
\caption{The orbit of $(x,y) = (\pi/10, 0)$ in the planar extension of the symmetrized Rosen map, $q=8$ --- 20000 points plotted.}
\label{extNatRosSymFig}
\end{center}
\end{figure}
 
Associate to the piecewise linear fractional map $h$ the two-dimensional function  $\mathcal T_h(x,y) = \mathcal T_M(x,y)$ whenever $h(x) = M\cdot x$.  In this subsection, we indicate how to pass from a  plot of the orbit of a point under $\mathcal T_h$ to the exact determination of $\mathcal E$, the domain fibering over $J$ on which this function is bijective.

 From the simple shape of the matrices involved, equation \eqref{eq:2Daction} becomes 
\[\mathcal T_{S^{-j} T}\,:(x, y) \mapsto  (\, S^{-j} T \cdot x,\, x^2 y - x\,)\,.\] 
Since $h(-x) = - h(x)$ almost everywhere, we find that the function $\mathcal T_h(-x,-y) =- \mathcal T_h(x, y) $ almost everywhere, and hence that $\mathcal E$  must be symmetric with respect to the origin.      

Figure ~\ref{extNatRosSymFig} allows one to guess that $\mathcal E$ is bounded by $y = 1/(x+1)$ above  $(0, \lambda/2)$ and by $y=1/(x-1)$ below $(-\lambda/2, 0\,)$.   From these, we solve for the boundaries of $\mathcal E\,$.

 We denote by  $\mathcal E_j$ the subset of  $\mathcal E$ fibering over the cylinder $\Delta_j$.    Numerical experiments show that the region where $x<0$ (and thus the various $j$ are positive)  is mapped by $\mathcal T_h$ to lie above the $x$-axis.   Similarly,  the right hand side has image below the $x$-axis.

Therefore, with our assumption that $\mathcal E$   lies within the bounds $y = 1/(x \pm 1)$,   keeping Equation \eqref{eq:TMonDelta} in mind, we find that  that the $\mathcal T_h$-image of $\mathcal E_j$  lies above that of $\mathcal E_{j+1}\,$.    Assuming that $\mathcal E$ is connected, we can thus solve for the upper boundary $y = 1/(x - \delta)$ of $\mathcal E_j$ when $j \ge 2$ --- since $\mathcal T_{S^{-j}T}$ 
sends $y = 1/(x+1)\,$ to $y = 1/(x + j  \lambda + 1)\,$ solving for equality with the image of $y = 1/(x - \delta)$ by $\mathcal T_{S^{-j-1}T}$ gives $\delta = -1/(\lambda - 1)\,$.   

 To solve for the various pieces of the upper boundary of $\mathcal E_1\,$,    as in \cite{BKS},  we use the remarks above on the orbit of $-\lambda/2=\phi_0$. Suppose that the upper boundary above $[\phi_j, \phi_{j+1}\,)$ is of the form $y = 1/(x - \delta_j\,)\,$.  Since $\mathcal T_{S^{-1}T}$ sends these boundaries one to the other, we find that $\delta_{j+1} = S^{-1}T\cdot \delta_j = - \lambda -1/\delta_j\,$.   Since we already have found that $\delta_{n-2} = -1/(\lambda-1)\,$, all of these values are determined.   Indeed,  we have that $S^{-1}T \cdot \delta_{n-2} = -1\,$, and thus $\delta_0 = (TS)^{n-1}\cdot -1\,$.   Now,  from the proof of Lemma 3.1 of \cite{BKS} (beware:  there one uses $q = 2p\,$), one has that   $(ST)^n \cdot 1 = -1\,$.  Since  $\delta_0 =  (TS)^{n-1}\cdot -1  = T (ST)^{n-1}T \cdot -1 = T(TS^{-1})(ST)^n\cdot 1\,$, we conclude that $\delta_0 = - \lambda - 1\,$.

\begin{figure}
\begin{center}
\scalebox{.4}{
\includegraphics{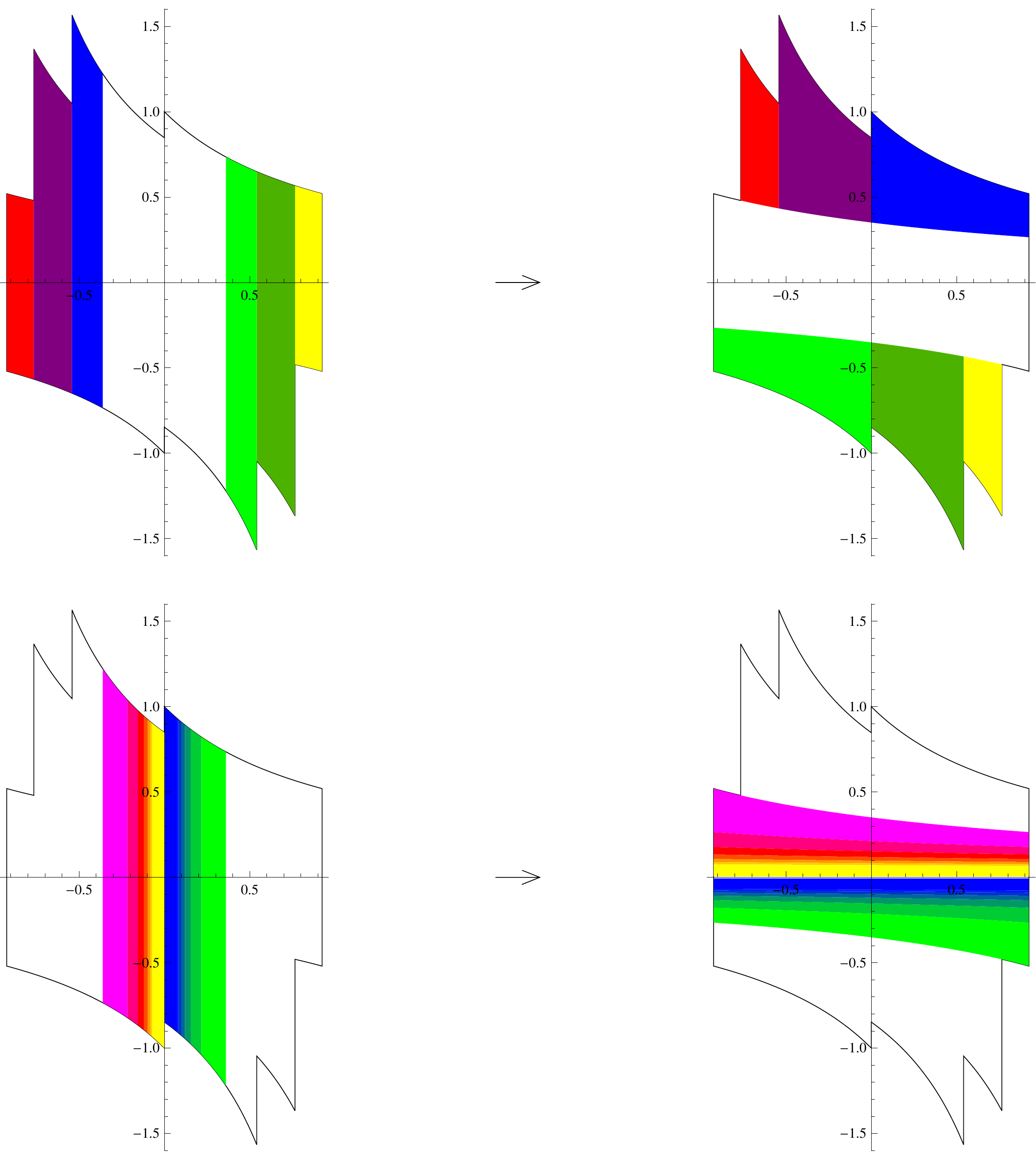}}
\caption{Cylinders of the planar extension and their images, $q=8$.}
\label{cylImageRosenFig}
\end{center}
\end{figure}
 
 In summary, we find that the dynamics of $\mathcal T_h$ on the left hand side of $\mathcal E$ can be described as follows.   For $j\ge 2\,$, the strip $\mathcal E_j\,$ ---   of  vertical sides $x = -2/(\, ( 2 j -1) \lambda)$ and $x = -2/(\, (  2 j + 1) \lambda)$  and curvilinear upper boundary $y = 1/(x + 1/(\lambda-1)\,)$  and lower boundary $y = 1/(x-1)$ ---  is sent to the curvilinear rectangle whose boundaries are of equation $x = \pm \lambda/2\,$,   $y = 1/(x + j \lambda +1)\,$,  and $y = 1/(x + (j -1) \lambda +1)\,$.      
 
 The subset $\mathcal E_1$ is more complicated.   The vertical boundaries are $x = -\lambda/2$ and $x = -2/3\lambda\,$.   The lower boundary is $y = 1/(x-1)$ and the upper boundary is given as above by the various $y = 1/(x-\delta_i)\,$.   The image of $\mathcal E_1$ consists of a curvilinear rectangle with boundaries $x=0$, $x= \lambda/2$ and the two hyperbolas $y = 1/(x+1)$ and $y = 1/( x + \lambda + 1)\,$, as well as an image on the left side of $\mathcal E$, that lies between $x =\phi_1= 2/\lambda - \lambda$ and $x=0$ and above $y = 1/(x + \lambda + 1)\,$, the upper boundary is given by $n-2$ hyperbolic segments.

Symmetry explains the remainder of the dynamics.   One can now easily verify that $\mathcal T_h$ does act so as to send $\mathcal E$ bijectively to itself (modulo, as usual,  sets of measure zero corresponding to boundaries of cylinders).   The upper part of Fig. \ref{cylImageRosenFig} shows the cylinder $\mathcal E_1$ and its image, in the case $q=8$, where the cylinder is divided in 3 parts with different upper boundaries; the lower part shows the cylinders $\mathcal E_j\,$ of $j>2$ and their images; the cylinders of negative index and their images are deduced by symmetry.

Note that we find the Markov condition of the continued fraction development --- there can be no more than $n-1$ successive partial quotients of value 1 (or -1), and if this limit is attained, then the next partial quotient is of the opposite sign.   These constraints are directly related to the relations which hold amongst our generators of the group $H_q\,$.

\section{Veech algorithm}\label{s:VeechFracs}
\subsection{Background}   W.~Veech showed \cite{V, V2} that the translation surface obtained by identifying, via translation, opposite sides of the regular Euclidean $q=2n$-gon  has a non-trivial group of affine diffeomorphisms.   In local coordinates, these functions are given by affine maps of the plane $v \mapsto Av + b$ with $A \in \text{SL}(2, \mathbb R)$, and furthermore the matrix part $A$ is constant throughout the collection of local coordinates giving the atlas of the translation surface.    The Veech group, $V_q\,$,  is the group generated by these matrix parts, which we will considered projectively.  Thus, $V_q \subset \text{PSL}(2, \mathbb R)\,$.    

Letting 
\[\mu := \mu_q = 2 \cot \pi/q\,,\]
  this group $V_q$ is generated by the parabolic $P = \begin{pmatrix} 1 & \mu\\0&1\end{pmatrix}$ and the rotation $R$ of angle $2 \pi/q\,$.  It is clear that $R$ induces an automorphism of the surface; perhaps less clear is the fact that $P$ also arises from an affine diffeomorphism.   Figure ~\ref{parabOctFig} suggests how in fact the image under $P$ of the regular polygon can be cut into a finite number of pieces which can then translated back, respecting identifications, to reform the polygon.   (In fact, $P$ arises from taking Dehn twists in appropriate parallel cylinders decomposing the surface.)

\begin{figure}
\begin{center}
\includegraphics{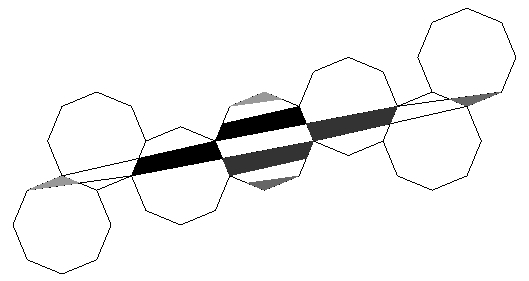}
\caption{The action of a parabolic element,  $q=8$}
\label{parabOctFig}
\end{center}
\end{figure}

Arnoux and Hubert \cite{AH} used the action of the Veech group on the set of foliations by parallel lines to define a family of additive and multiplicative continued fraction algorithms for developments of real numbers.   These algorithms are related to the Teichm\"uller geodesic flow, as explained in \cite{AH}.   

Their algorithms are given in terms of the inverse of the slope of the foliation.  Here, it is more convenient to study the same action of the Veech group on the foliations, but rather in terms of the slopes themselves.  We next give a brief sketch of the action and describe a natural algorithm.

\begin{rem}
Warning --- note that the standard action of a matrix $\begin{pmatrix}a & b\\c&d\end{pmatrix}$ sends a line of direction vector $\begin{pmatrix} 1  \\y\end{pmatrix}$ to a line of direction vector $\begin{pmatrix} 1 \\\frac{c + d y}{a + by}\end{pmatrix}\,$.  Thus, the matrix $\begin{pmatrix}a & b\\c&d\end{pmatrix}$  acts on real slopes $z$ by the standard fractional linear transformation of $\begin{pmatrix}d & c\\b&a\end{pmatrix}\,$!  Fortunately, this twisting is simplified in the setting of $V_q\,$, in particular,  the action of $R$ on slopes is given by the M\"obius action of $R^{-1}$ on real values;   similarly,  the M\"obius action of the transpose of $P$ (which is conjugate in $V_q$ to $P$, see \cite{AH}) gives the geometric action of $P$ on slopes.   
\end{rem}

\subsection{Additive algorithm}    We first give a purely geometric description of the algorithm.   Consider a regular polygon in the complex plane with $q=2n$ sides,  centered at zero, and with a vertex at the point $i \in \mathbb C\,$.  

\begin{figure}
\begin{center}
\scalebox{.7}{\includegraphics{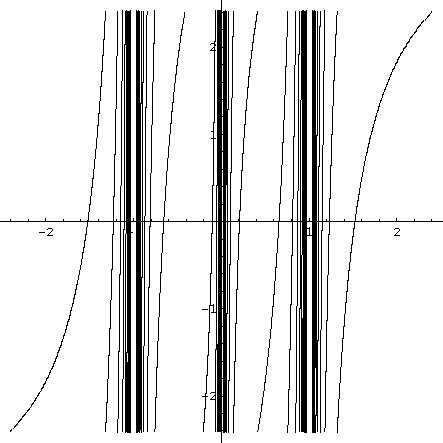}}
\caption{Representation of the graph of the additive Veech map, $q=8$}
\label{addVeechFig}
\end{center}
\end{figure}

As an element of the projective group, the rotation $R$ is of order $n\,$; since the perpendicular bisectors of the sides adjacent to the vertex at $i$ have slope $\mu/2$ and $-\mu/2$, for each real slope $x$ of absolute value at most $\mu/2\,$, there is exactly one  power (modulo $n$) of $R$ such that $R^j\cdot x$ is not contained in the interval $(\, -\mu/2, \mu/2\,)\,$ (indeed, this is the power of $R$ such that the image of the line of slope $x$ passes closer to the vertex $i$ than the image under all other powers of $R\,$).     Of course, for any real $x$, there exists a unique $k$ such that $P^k \cdot x \in [\, -\mu/2, \mu/2\,)\,$.   

The {\em additive algorithm} is defined on $I := I_q = [-\mu/2, \mu/2\,]\,$; for a given $x$, we apply the power of $R$ that sends $x$ outside of $I$, and then apply the unique power of $P$ that brings this value back into the interval.   

We now more precisely give the corresponding function.    Let $d_j := R^{-j}\cdot \frac{\mu}{2}$ for $j = 1, \dots, n\,$, and $c_j := R^{-j}\cdot \infty\,$.   One has
\[-\infty < d_1 = -\dfrac{\mu}{2}< c_1 < d_2 < \, \cdots\,  < c_{n-1} < d_n = \dfrac{\mu}{2} < \infty\,.\]
The application is defined on $I = (d_1, d_n)$ except for the countable number of points of discontinuity (including in particular the $d_j$ and $c_j$). On each subinterval $(d_j, c_j)\,$, the function is given by $x \mapsto P^k R^j\cdot x$ where $k$ is a strictly negative integer (depending on $x$);  on $(c_j, d_{j+1})\,$,  the function is $x \mapsto P^k R^j\cdot x$ where now $k$ is a strictly positive.   

We thus define the {\em partial quotients} $(k,j)$ and use $\Delta(k,j)$ to denote the corresponding cylinder, thus the subinterval on which the function is given by applying $P^k R^j\,$.       Here, each $k$ is non-zero, and $j$ is in $\{1, \dots, n-1\}\,$.      Note that here all of the cylinders are full and the map is clearly Markov.  This reflects the fact that the elements $R$ and $P$ generate cyclic groups whose free product gives all of $V_q\,$.

The standard ordering of the reals $x \in \bigcup\, \Delta(k,j)$  induces an ordering on the partial quotients $(k,j)$ as follows: 
\[
\begin{aligned}
&(-1, 1) < (-2,1) < \cdots < (-\infty, 1)< (\infty, 1) < \cdots < (2,1) < (1,1) \\&
< (-1, 2) < (-2,2) < \cdots<  (-\infty, 2) < (\infty, 2) < \cdots < (1,2)\\
&< (-1,3) < \cdots < (-1, n-1) < (-2, n-1) < \cdots< (-\infty, n-1) \\
&< (\infty, n-1) < \cdots < (2, n-1) < (1, n-1) = (1, -1)\;.
\end{aligned}
\]

For explicit calculations, it is helpful to have $R$ in terms of $\mu\,$.   Applying the classical formulas relating the sine and cosine to the tangent of the half-angle, one has 

\begin{equation}
R = \dfrac{1}{\mu^2 +4}  \, \begin{pmatrix} \mu^2 -4&-4\mu\\
                                       4\mu &\mu^2 -4\end{pmatrix}\;.
\end{equation}                                       

The action of $R$ on the endpoints of the interval is of fundamental importance for the following calculations.  By direct geometric argument (beware of the change in actions requiring a passage from $R$ to $R^{-1}$ !), or simply by direct calculation one finds
\begin{equation}
R \,\cdot(- \mu/2) =  \mu/2\,;\;\; \;\;R^{-1}\,  \cdot  \mu/2 =  -\mu/2\;.
\end{equation}

The additive Veech map has an infinite invariant measure, we thus ``accelerate'' it to give a ``multiplicative'' algorithm. 

\subsection{Multiplicative algorithm}     The reason that the additive map has an {\em infinite} invariant measure is that it has two (indifferent) fixed points: each of the parabolic elements $P^{-1}R$ and $P R^{-1}$ is applied in a subinterval  whose boundary contains its fixed point.   We accelerate the map by appropriately grouping together powers of each of these parabolic elements and thus define our   multiplicative Veech map, $v : I \to I\,$.     

\subsubsection{Acceleration near $\mu/2$}   The fixed point of $PR^{-1}$ is $\mu/2\,$.   Let 
\[ \alpha := R P^{-1} \cdot (- \mu/2) = \dfrac{\mu}{2}\; \dfrac{3 \mu^2 - 4}{5 \mu^2 + 4}\,.\]
Thus $\alpha$ is the left endpoint of $\Delta(\,(1, -1)\,)$ the domain of $P R^{-1}$ for the additive algorithm.   We now define the image of $x \in (\alpha, \mu/2]$ under the multiplicative map as $(P R^{-1})^{\ell}\cdot x\,$, with $\ell$ the smallest power such that $(P R^{-1})^{\ell}\cdot x$ lies outside of $(\alpha, \mu/2]\,$.

To determine this power $\ell$, it is simpler to conjugate by a matrix that sends $0$ to $\mu/2$.   Let  $C = \begin{pmatrix} 1   & \mu/2 \\   0 & 1\end{pmatrix}$, and consider the matrix $C^{-1} PR^{-1} C$; it is a parabolic element which preserves 0. Computation shows that, in the projective group, it is equal to   $D = \begin{pmatrix} 1   & 0  \\  
 4 \mu/(\mu^{2}+4)  & 1\end{pmatrix}$.   We look for the smallest  $\ell$ such that $(P.R^{-1})^{\ell}.x$ is smaller than $\alpha$, which is equivalent to   $D^{\ell} .(x-\mu/2 ) <  \alpha-\mu/2$.  The smallest such $\ell$ is 

\begin{equation} 
\ell = \bigg\lceil  \, \frac{\mu^{2}+4}{\mu}  \,\frac{1}{\mu-2 \alpha} 
 \,\frac{x-\alpha }{\mu-2 x }\bigg\rceil \; . 
\end{equation}
Thus, we define $v(x) = (P R^{-1})^{\ell}\cdot x$ with $\ell$ as above, for any  $x \in (\alpha, \mu/2]\,$.
 
\subsubsection{Acceleration near $-\mu/2$}   The additive Veech function being odd, it is clear how we proceed here.  For $x \in [\, - \mu/2, - \alpha\,)\,$ we let $v(x) = (P^{-1}R)^k\,$ where $k(x)=\ell(-x)$, that is
\begin{equation} 
k = \bigg\lceil  \, \frac{\mu^{2}+4}{\mu}  \,\frac{1}{\mu-2 \alpha } 
 \,\frac{-\alpha - x}{2 x + \mu}\bigg\rceil \; . 
\end{equation}

\subsubsection{No acceleration in central section}   For $x \in [-\alpha, \alpha)\,$,  the multiplicative and additive maps are the same.     Thus,  $v(x) = P^k R^j \cdot x$ in this subinterval, with these exponents as in the previous subsection.  \\

The graph of the multiplicative function is given in Figure ~\ref{multVeechFig}.  The function has an infinite number of intervals of continuity, accumulating at the endpoints $\pm  \mu/2\,$ and at  $n-1$ further points in the $V_q$-orbit of infinity.   As the figure indicates,  these intervals form a countable Markov partition: indeed,  $-\alpha$, $\alpha$ are points of discontinuity, and each interval of continuity has image one of the three intervals:   $I = [\,-\mu/2, \mu/2)$,  $[\,-\mu/2, \alpha)$,  and $[\, -\alpha, \mu/2)\,$.

\begin{figure}
\begin{center}
\scalebox{.7}{\includegraphics{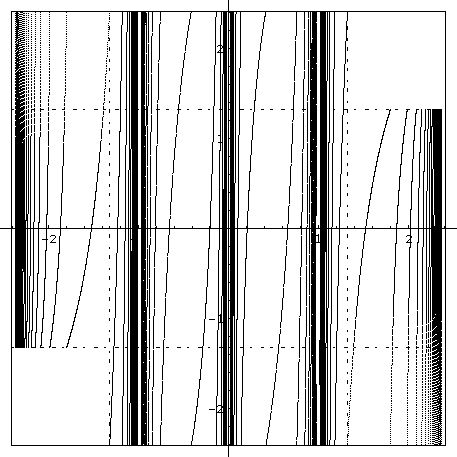}}
\caption{Representation of the graph of the multiplicative Veech map, $q=8$}
\label{multVeechFig}
\end{center}
\end{figure}
 
\subsection{Planar transformations for the Veech maps}   Because of its straightforward combinatorics,  the planar extension of the additive algorithm is particularly easy to find.   Computational experiment (see Figure ~\ref{natExtVeechAddFig})   shows that the domain is bounded by the graphs of M\"obius transformations, each with a pole at an endpoint of $I\,$.

 One is thus lead to conjecture that the domain $\Omega_a$ of planar extension of the additive map is 
 \[
 \Omega_a = \{ \,(x,y)\;\vert\, x \in I\,,\, \dfrac{1}{x - \mu/2} < y <  \dfrac{1}{x + \mu/2} \,\}\,.
 \]
 It is easy to prove that the planar extension map is bijective on this domain.    Let  $\mathcal D(k,j)\subset \Omega_a $ be the region projecting to the cylinder $\Delta(k,j)\,$.   Thus, restricted to $\mathcal D(k,j)\,$, the planar extension map is given by $\mathcal T_M$ with  $M = P^k R^j\,$.  Using Equation~\eqref{eq:TMonDelta}, one verifies that 
 
\[
 \begin{matrix}  P^{k}R^{j} \;\;\text{sends} \;\;\;   y = \dfrac{1}{x + \mu/2}   \; \;\text{to}\; \;\; y = \dfrac{1}{x - k \mu - 
R^{j} .\frac{-\mu}{2}} \\
\\
 P^{k}R^{j} \;\;\text{sends} \;\;\;  y = \dfrac{1}{x - \mu/2}   \; \;\text{to}\; \;\; y = \dfrac{1}{x - k \mu - 
R^{j} .\frac{\mu}{2}}  \end{matrix} \;.  
\]
 Now,  the orbit of $\mu/2$ under powers of $R$ remains in $I$ and as the powers $k$ above are nonzero, $k \mu + R^j \cdot \dfrac{\mu}{2}$ and  $k \mu + R^j \cdot \dfrac{-\mu}{2}$ are always outside of $I$, and of sign that of $k\,$.  That is,  the two dimensional cylinder $\mathcal D(k, j)$ is sent to the upper half-plane when $k$ is negative, and to the lower half-plane when $k$ is positive.

 In the case that $k <0\,$,  the image of $\mathcal D(k, j)$ is bounded above by $ y = \dfrac{1}{x - k \mu - 
R^{j} .\frac{-\mu}{2}}$ and bounded below by  $y = \dfrac{1}{x - k \mu - 
R^{j} .\frac{\mu}{2}} \,$.     Since $R \cdot \frac{-\mu}{2} = \frac{\mu}{2}\,$,  whenever $j< n-1$ the lower boundary of the image of $\mathcal D(k, j)$ is the upper boundary of the image of $\mathcal D(k, j+1)\,$.  If $j = n-1\,$, we have 
\[
\begin{aligned}
k \mu + R^{n-1}\cdot \frac{\mu}{2} &= k \mu + R^{-1}\cdot \frac{\mu}{2} \; = k \mu +  \frac{\mu}{2}\\
                                                       &= (k+1) \mu -  \mu/2 = \;  (k+1) \mu + R\cdot \frac{-\mu}{2}\,,
\end{aligned}
\]
thus the lower boundary of the image of $\mathcal D(k, n-1)$ is also the upper boundary of the image of $\mathcal D(k+1, 1)\,$.    In summary, when $k <0\,$, the images of these two dimensional cylinders are stacked with respect to the vertical in increasing order in accordance with the following ordering of indices.
\[
\begin{aligned}
(-1,1) < (-1,2) < \cdots (-1, n-1) < (-2, 1), (-2, 2) < \cdots < (2, n-1)\\
 < (-3, 1) < \cdots < (k, 1) < (k,2) < \cdots < (k, n-1) < (k-1,1) < \cdots \;.
\end{aligned} 
\]
Note that in particular the upper boundary of the image of $\mathcal D(-1,1)$ is $y = 1/(\, x + \mu/2\,)$ and its lower boundary is given by $y = 1/(x + \gamma)\,$ with $\gamma = \mu - R \cdot \frac{\mu}{2}\,$.  This cylinder is of course of infinite (Lebesgue) measure.  

\begin{figure}
\begin{center}
\scalebox{.9}{\includegraphics{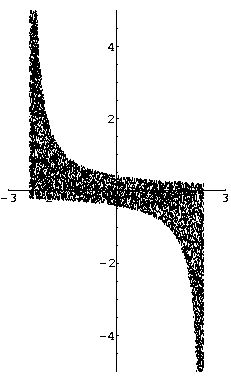}}
\caption{The orbit of $(\pi/10, 0)$ in the planar extension of the additive Veech map, $q=8$ --- 20000 point plotted.}
\label{natExtVeechAddFig}
\end{center}
\end{figure}

By symmetry, one completes the study of the dynamics of the planar extension; in particular,  the cylinder $\mathcal D(1,n-1)$ --- corresponding to the second parabolic fixed point ---  is also of infinite measure. 
 
 The planar extension for the multiplicative algorithm is now easy to deduce.   Indeed, this is the same as for the additive algorithm, except that at each point of the two cylinders $\mathcal D(-1,1)$ and $\mathcal D(1,n-1)$   we must iterate the corresponding $\mathcal T_M$ so as to exit the cylinder.    We find that the domain of $\mathcal T_v$ has each of its upper and lower boundaries of a single point of discontinuity, of $x$-coordinate $-\alpha$ and $\alpha\,$, respectively.   
 
 Recall that $\alpha = \dfrac{\mu}{2}\; \dfrac{3 \mu^2 - 4}{5 \mu^2 + 4}$ and let $\gamma = PR^{-1} \cdot \frac{\mu}{2} = \dfrac{\mu}{2}\; \dfrac{5 \mu^2 + 4}{3 \mu^2 - 4}\,$.   We have thus have the following. 

\begin{prop}\label{p:VeechMultNat}  Let $q = 2 n$ be an even integer greater than 4, and let $\mu = 2 \cot \pi/q\,$.  Let $v(x)$ be the function on $I = [\, - \mu/2, \mu/2\,]$ defined  by the multiplicative Veech algorithm of index $q$ on slopes.   

Let $b_{+}, b_{-}$ be the two functions defined by 
\[
\begin{cases}
b_{+}(x) = \frac{1}{x + \gamma}&\text{if}\;\; x \in [-\mu/2,-\alpha)\,;\\
b_{+}(x) = \frac{1}{x + \mu/2}&\text{if}\;\; x \in [-\alpha, \mu/2\,]\,;\\
b_{-}(x) = \frac{1}{x - \mu/2}&\text{if}\;\; x \in [-\mu/2,\alpha)\,;\\
b_{-}(x) = \frac{1}{x - \gamma}&\text{if}\; \; x \in [\alpha, \mu/2\,]\;.
\end{cases}
\]
The planar extension $\mathcal T_v$ of $v(x)$ is defined on the domain 
\[ \Omega_v = \{\,(x,y)\,\vert\, x \in I, \, b_{-}(x) < y < b_{+}(x)\,\}\,.\]    The area of $\Omega_v$ is  $c_v := 2 \, \log 8 \cos^2\, \frac{\pi}{q}\,$. 
\end{prop}

\begin{proof}   All that remains to justify is the value of $c_v\,$.  We have
\[
\begin{aligned}
c_v &= 2 \;\left( \int_{0}^{\alpha} \dfrac{1}{x+\mu/2}- 
\dfrac{1}{x-\mu/2}\,dx +
\int_{\alpha}^{\mu/2} \dfrac{1}{x+\mu/2}- \dfrac{1}{x-\gamma}\,dx 
\right)\\
\\
&= 2 \;\log \left( \dfrac{\alpha + \mu/2}{\mu/2}\; 
\dfrac{\mu/2}{\mu/2-\alpha}\;\dfrac{\mu}{\alpha+\mu/2}\;
\dfrac{\gamma-\alpha}{\gamma-\mu/2}\right)\\
\\
&= 2 \; \log \dfrac{\mu\, 
(\gamma-\alpha)}{(\mu/2-\alpha)(\gamma-\mu/2)}\\
\\
&= 2\;  \log  \dfrac{\mu\, (\mu/2)(\, \frac{5\mu^2 +4}{3 \mu^2 
-4}-\frac{3 \mu^2 -4}{5 \mu^2 + 4}\,)} 
                               { (\mu/2)^2(\, 1-\frac{3 \mu^2 -4}{5 
\mu^2 + 4}\,) (\, \frac{5\mu^2 +4}{3 \mu^2 -4}-1\,)}\\
\\
&= 2\;  \log  2\,\dfrac{(5\mu^2 +4)^2 - (3 \mu^2 -4)^2} 
                               { [\, 5 \mu^2 + 4-(3 \mu^2 -4)\,]^2}\\
\\
&= 2\;  \log  \,\dfrac{8\mu^2} 
                               {  \mu^2 +4} \;= \; 2\;  \log  \,2\mu^2\, \dfrac{4} 
                               {  \mu^2 +4}\\
\\
&= 2\;  \log  \,8\cos^2\frac{\pi}{q}\,.
\end{aligned}
\]

\end{proof}

\begin{figure}
\begin{center}
\scalebox{.9}{\includegraphics{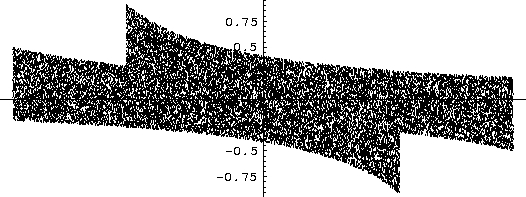}}
\caption{The orbit of $(\pi/10, 0)$ in the planar extension of the multiplicative Veech map, $q=8$ --- 20000 point plotted.}
\label{natExtVeechMultFig}
\end{center}
\end{figure}

\section{Conjugation of the Rosen algorithm}\label{s:conjugRosen}
In order to directly compare the symmetric Rosen algorithm with our Veech algorithm, the groups generated by the fractional linear transformations giving the respective maps should be the same.   That is not the case, so we first make a conjugation to come closer to this equality.    However, the resulting interval map is defined on $[\, 0, \mu\,]\,$, thus we cut and translate appropriately so as to find an interval map defined on $I\,$.     

\begin{figure}
\begin{center}
\scalebox{.9}{\includegraphics{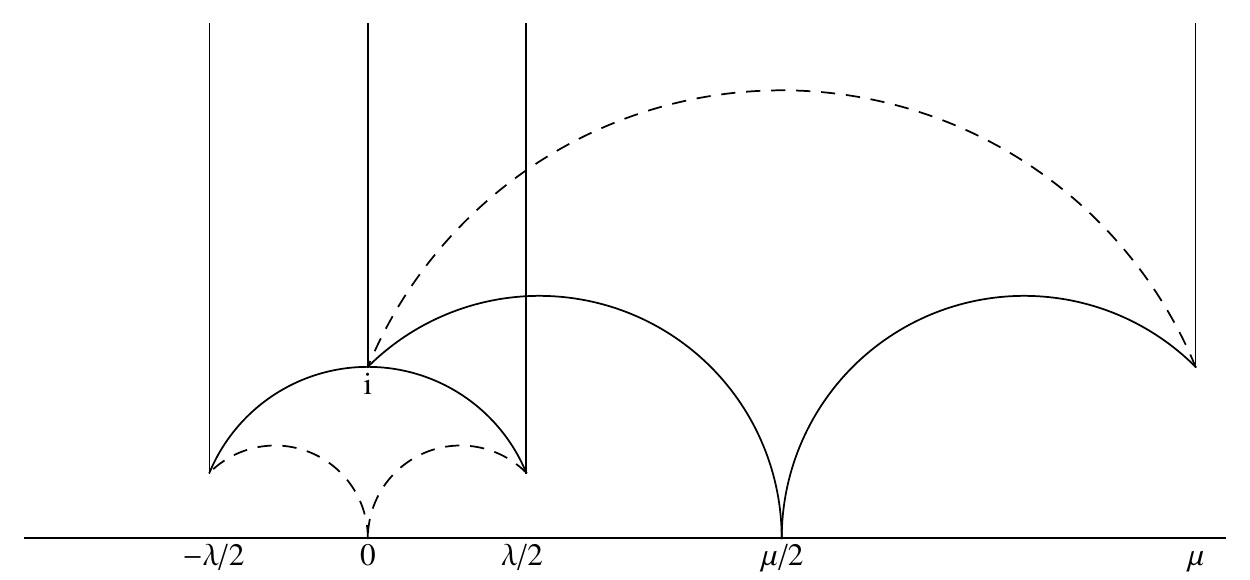}}
\caption{The fundamental domains for the Hecke group and the Veech group}
\label{domfondFig}
\end{center}
\end{figure}

\subsection{The conjugated Rosen map}

It is easy to prove that the Hecke group, acting on the upper half-plane, admits a fundamental domain bounded by the vertical lines $x=\pm \lambda/2$ and the unit circle centered at the origin; it contains a subgroup of index two with a fundamental domain bounded by the same vertical lines, and two circles of radius $1/\lambda$ and center $(\pm 1/\lambda,0)$ (see Fig. \ref{domfondFig}). On the other hand, the Veech group admits a fundamental domain bounded by the vertical lines $x=0$ and $x=\mu$, and two circles of radius $1/\sin(2\pi/q)$ and center $(\mu/2\pm 1/\sin(2\pi/q),0)$; from this, one obtains easily a conjugacy between the Veech group and the subgroup of the Hecke group.

More explicitly, the Fuchsian group $V_q$ generated by  $P = \begin{pmatrix} 1   & \mu \\  
              0 & 1\end{pmatrix}$ and $R$ is an index two subgroup of the group generated by $P$ and $U$, the rotation of angle $\pi/q$, which is   $\text{PSL}(2, {\mathbb R})$- conjugate to the Hecke group of index $q\,$.   As shown in \cite{AH}, this conjugation is given by the matrix  
\begin{equation}\label{eq:conjMatM}
    Q = \frac{1}{\sqrt{\sin\, \pi/q}} \begin{pmatrix}  1   & \cos\, \pi/q \\  
                                               0 & \sin\, \pi/q 
\end{pmatrix}\,.
\end{equation}
  Note that $Q$ sends $[-\lambda/2, \lambda/2)$ to   $ [\,0, \mu\,)$.    Thus,  in order to transfer the symmetric Rosen map to 
the interval 
$ I = [\,-\mu/2, \mu/2\,]\,$,  for any $x \in [0, \mu/2)$ we simply use $x \mapsto Q\cdot  h(Q^{-1}\cdot x)$;  
for those $x \in  [\,-\mu/2, 0)$ we first apply $x \mapsto P\cdot x = x + \mu$, and then act in the analogous way.  

To discuss this explicitly, we use the following easily checked identities.
\begin{lem}\label{l:conjugates}   The  following equalities hold.
\begin{enumerate}
\item $Q SQ^{-1} = P$;
\item $QTSQ^{-1} = U$;
\item $QTSTQ^{-1} = RP^{-1}$.
\end{enumerate}
\end{lem}
 
 Recall that $T \cdot x=  -1/x\,$.  We have the following ``conjugated'' version of the symmetric Rosen map.

\begin{lem}\label{lemRosenSymOnI}
  Let $\widetilde T := Q T Q^{-1}\,$.    Then the map, as described above, induced on $I$  by the symmetric Rosen algorithm is given by 
  \[
\begin{aligned} 
&h_{Q}: \;\; [-\mu/2, \mu/2\,) \;\;\;  \to  \;\;\;   [-\mu/2, 
\mu/2\,)\\
 \\
                      &x     \mapsto 
		        \begin{cases}  \widetilde T  \cdot x - \lfloor \, (\widetilde T  \cdot x)/\mu + 
		        1/2\, \rfloor \,\mu  &
		                 \text{if}\; \; x \in (\, 0, \mu/2\,)  ;\\
				 \\
		          \widetilde T P \cdot x - \lfloor\,  (\widetilde T P \cdot x)/\mu + 
		        1/2\, \rfloor \,\mu  &
			         \text{if}\;\;  x \in [\, -\mu/2, 0\,) \; .
			\end{cases} 
\end{aligned} 			
\]

\end{lem}
  
Although the map $h_Q$ may seem to arise artificially, it is in fact easily described geometrically. 
\begin{defin} Let  $k$  be the interval  map on $I$ defined by the following geometric algorithm:

 --- a positive  $x\in I$  is sent to its image under the rotation   $U^{-1}$, if this lies outside of $I$ then the final image is given by applying the exact power of $P$ that brings it back into $I\,$;
 
---  a negative  $x\in I$  is sent to its image under the rotation   $U$, if this lies outside of $I$ then the final image is given by applying the exact power of $P$ that brings it back into $I\,$.
\end{defin}

\begin{prop}\label{propRosenSymIsGeom}   The interval maps  $h_Q$  and  $k$ are equal.   
\end{prop} 
\begin{proof} We have 
\[
\widetilde T P = QTQ^{-1} P  =  QTQ^{-1} QSQ^{-1}  =  U\,.
\]
The equality of the two functions on $[-\mu/2, 0)$ thus holds.   

Since $U^{-1} = P^{-1} \widetilde T$   and for each $x\ge 0$ there is an integral $j$ such that $h_q(x) = P^{j} \widetilde T \cdot x$, the equality also holds on $(0, \mu/2)$. 

\end{proof}

\begin{rem}    Despite the obvious similarities of this algorithm with  the additive Veech algorithm,
whereas the  additive Veech map is of infinite invariant measure,   this version of
the Rosen map  is of finite invariant measure.
\end{rem}

\subsection{Doubling the conjugated Rosen map}
 We now consider the function  $r(x) := k^2(x)$.  This is a map that we can directly compare with the multiplicative Veech map, $v(x)\,$.    One easily finds that  $r(x) = R^{-1} \cdot x$ when $x \in [\,0, R \cdot \mu/2\,)$ and also that  $r(x) = R \cdot x$ for 
$x \in [\,R^{-1} . (-\mu/2)\,, 0\,)$. 
 
We next explicitly give $r(x)$ for negative values, symmetry allows the reader to extend this to all of the interval $I\,$.
\begin{prop}    For $x \ge 0\,$ the map $r(x)$ is given as follows. 

\[
r(x) = \begin{cases}  R^{-1} \cdot x  & \text{if}\;\; x\in [0, R \cdot \mu/2)\;;\\
\\
P^l (PR^{-1})^{k} \cdot x  & \text{if}\; \;x\in 
(RP^{-1})^{k}P^{-l}(\;[\;-\mu/2, \mu/2)\;)\;;\\
\\
(PR^{-1})^{k} \cdot  x& \text{if}\;\; x \in (RP^{-1})^{k}(\;[-\mu/2, R \cdot 
0)\;)\;;\\
\\
R^{-1}(PR^{-1})^{k} \cdot x & \text{if}\;\; x \in  
(RP^{-1})^{k}R(\;[R^{-1} \cdot  0,  \mu/2)\;)\;;\\
\\
P^{-j}R^{-1}(PR^{-1})^{k} \cdot x\;\;\;\; & \text{if}\;\; x \in  
(RP^{-1})^{k}RP^{j}([\;-\mu/2, \mu/2)\;)\;,
\end{cases}
\]
with $j,k,l \in {\mathbb N}\,$.
\end{prop}  
\begin{proof}
   We treat several cases, determined by the value of $k(x)$.  

\noindent
{\bf case 1 :   $\;\;k(x) \in [0, U \cdot \mu/2)$}

In this case, there are three possibilities, according to the value of $x$: 

\[
r(x) = \begin{cases}  R^{-1} \cdot x  \;;\\
\\
                      U^{-1}P^{-k} U^{-1} \cdot x  \;;\\
\\
 U^{-1}P^{l} U \cdot x  \;.
           \end{cases}
\]
   
Now, using that fact that  $T^2 = \text{Id}$ and the identities of Lemma ~\ref{l:conjugates}, we find that the final two of these possibilities are also as claimed. 
\[
\begin{aligned} 
 U^{-1}P^{-k} U^{-1} &= Q S^{-1}TQ^{-1}  \cdot QS^{-k}Q^{-1}  \cdot Q 
S^{-1}T M^{-1}\\
\\
                                           &= Q S^{-1}TS^{-1}T 
TS^{-k} T Q^{-1}\\
\\
                                           &=  Q S^{-1}TS^{-1}T 
Q^{-1}  \cdot TS^{-k} T Q^{-1}\\
\\
                                           &= R^{-1} (PR^{-1})^k\\
\\
                                           &= P^{-1} 
(PR^{-1})^{k+1}\;.
\end{aligned} 
\]

\medskip  
 \noindent
{\bf case 2 :   $\;\;k(x) \in [U^{-1} \cdot (-\mu/2)\,,0)$}

Here again there are three possibilities, and similar reasoning leads to 
\[
r(x) = \begin{cases}  R \cdot x  \;;\\
\\
                      UP^{-k} U^{-1} \cdot x  \;;\\
\\
 UP^{l} U \cdot x  \;.
           \end{cases}
\]

Now one verifies that  $UP^{-k} U^{-1} = 
(PR^{-1})^{k}$,  and that 
 $UP^{l} U = (RP^{-1})^{l}R\,$.

\medskip  
 \noindent
{\bf case 3 :   $\;\;k(x) \notin [U^{-1} . (-\mu/2)\,,U \cdot 
\mu/2)$}

If $k(x)$ is of sign differing from $x$, then $r(x)$ is either 
$P^{-k} U^{-1}P^{l} u\cdot x$  or $P^{l} UP^{-k} 
U^{-1}\cdot x\,$.  On finds that

\[
P^{-k} U^{-1}P^{l} U = P^{-k} (P^{-1}R)^l \;;\;\;\;  P^{l} 
UP^{-k} U^{-1}= P^l (PR^{-1})^k\;.
\]

If $k(x)$ has the same sign as $x$, then $r(x)$ is either 
$P^{-j} U^{-1}P^{-k} U^{-1} \cdot x$ or
$P^{m} UP^{l} U \cdot x\,$.  On finds that  
\[
\begin{aligned}
P^{-j} U^{-1}P^{-k} U^{-1}  &= P^{-(j+1)} 
(PR^{-1})^{k+1}\,;\\
 P^{m} UP^{l} U &= 
P^{m-1} (P^{-1}R)^{l+1}\,.
\end{aligned}
\]

The proof is completed by verifying that the domains of the pieces of the function are indeed the subintervals claimed.
\end{proof}

\subsection{The planar domain}
 From our definitions, we have that  $\mathcal T_r(x,y) = \mathcal 
T_{k}^{2}(x,y)$  and  as   $\mathcal T_k(x,y)$  is Lebesgue almost everywhere surjective, the domains of definition of these two transformations are the same.  
 Since we already have the domain of the planar map for the symmetric Rosen map,   the  calculation of the equations for the boundary of the planar domain  for $k(x)$ is rather straightforward. 
 
We first collect some identities.   Recall that   $\phi_j = f^j(-\lambda/2) = (S^{-1}T)^{j}\cdot(-\lambda/2)$;    $\delta_{j+1} = S^{-1}T\cdot \delta_j = - \lambda -1/\delta_j\,$,  with $\delta_0 = - \lambda - 1\,$ and thus $\delta_{n-2} = -1/(\lambda-1)\,$.

\begin{lem}\label{lemUsefulIds}   Let   $q=2n\,$, $\mu=\mu_q\,$ and  $Q$ the matrix of conjugation given in Equation~\eqref{eq:conjMatM}.     Then the following equalities hold.

\begin{enumerate}
\item $Q\cdot\phi_j =   \tan j \pi/q  \;\; \text{for}\; \; j\ge 0\;$;
\item $Q\cdot \delta_j = -U^{j} \cdot(Q\cdot1)\;\; \text{for}\; \; j\ge 0\;$;
\item $P^{-1}Q\cdot(-\delta_j) =  U^{j} \cdot(Q\cdot1) \;\;  \text{for}\; \; j\ge 
0\;$.
\end{enumerate}
\end{lem}

\begin{proof}   Since $\phi_0 = -\lambda/2$, direct calculation gives $Q\cdot\phi_0 = 0\,$.  Thus, 
\[ Q \cdot \phi_j = Q(S^{-1}T)^j\cdot (-\lambda/2) = Q(S^{-1}T)^jQ^{-1} \cdot 0 = U^{-j}\cdot 0\,.\]
 Since $U$ is the rotation of angle
$\pi/q$,  the first equality holds.

     Since  $\delta_0 = -1 -2 \cos \pi/q$,   direct calculation gives  $Q\cdot(\delta_0) = 
-Q\cdot 1\,$.  Since $\delta_j = (S^{-1}T)^j$, as for the previous identity, we find $Q\cdot(\delta_j) =  U^{-j} \cdot (- Q\cdot 1)$.   
But, $-(\Lambda\cdot( -x)\,) = \Lambda \cdot x$  for any rotation $\Lambda$  and any real $x$,  and the second equality holds.

 For the third identity,  the case of $j=0$ can be directly verified.   For $j>0$,  we rely on the fact that for any $y$, 
 $-(S^{-1}T \cdot y) = -(-\lambda -1/y) = ST\cdot(-y)$.  By induction, we find  $- (\, (S^{-1}T)^j \cdot y) = (ST)^j\cdot(-y)$. Therefore, 
 \[
 \begin{aligned}
  P^{-1}Q\cdot(-\delta_j)  &= P^{-1}Q (ST)^j\cdot (-\delta_0) = P^{-1}QS Q^{-1} \, Q  (TS)^j S^{-1}\cdot (-\delta_0))\\
                                        &=   U^j QS^{-1}\cdot (-\delta_0)) = U^j P^{-1}Q\cdot (-\delta_0))\,
  \end{aligned}
  \]
 and the result holds. 
\end{proof}

The reader may wish to compare the following description of the boundary of $\Omega_r$  with Figure~\ref{natExtOverlayFig}.
\begin{lem}\label{lemNatExtDouble}  Let $\Omega_{r}$  let be the domain of the planar system, determined as usual,  of $r(x)\,$.   Then the boundary of $\Omega_r$ is given by 
\[
\begin{cases}
y = 1/(x -  Q  \cdot 1\,) &  x \in [\,-2/\mu,\mu/2\,)\;;\\
\\
y = 1/(x + Q  \cdot 1\,)&  x \in [\,-\mu/2, 2/\mu\,)\;;\\
\\
y = 1/(x  - U^{j} Q  \cdot 1)\,)& x \in [\,-\tan(j+1)\pi/q,-\tan j \pi/q \,)\,\;;\\
\\
y = 1/(x +  U^{j} Q  \cdot 1)\,)&  x \in [\,\tan j 
\pi/q,\tan (j+1) \pi/q)\;,\\
\end{cases}
\]
\noindent
where $j\in\{1, \dots,  n-2 \}$.
\end{lem}
 
 \begin{proof}(Sketch)
The first step is the conjugation by the matrix $Q$.   The image of points  $(x,y= 1/(x-\
\delta)\,)$ being  $(Q \cdot x,  y=1/(Q \cdot x - Q \cdot 
\delta)\,)$,  we find that the boundary above $(\mu/2, \mu)$ 
is given by  $y = 1/(x + Q \cdot 1)$, and that under  $(0,\mu/2)$ 
is $y = 1/(x - Q \cdot 1\,)$; the remaining boundaries follow as easily.  

Secondly, we cut and translate by letting $P^{-1}$ act on the right half.    
Since the $(2,1)$-element of $P^{-1}$ is $c=0\,$,  one finds that only the $x$-coordinate values change.    Of course, the equations giving the new boundaries change accordingly.  
\end{proof} 

\begin{rem}   For the study of the intersections of the planar extensions of $r(x)$ and   $v(x)$ it is helpful to note that  after   $y = 1/(x+ Q \cdot 1)$ the next lowest piece of the upper boundary of $\Omega_{k}\,$ is of equation  $y = 
1/(x  - Q  \cdot (-\delta_1)\,\;)$.   Equivalently, this is $y=1/(x+U Q \cdot 1)\,)$.  Furthermore,  $Q  \cdot (-\delta_1) = Q  \cdot 
(-\lambda +1/(\lambda+1)\;) = \mu/2 +(-\mu +\rho)$, with $\rho = 1/(\; 
(\sin \pi/q)(\lambda + 1)\;)$.    Thus,  $-Q  \cdot (-\delta_1)  =  \mu/2  
- \rho$ and  $\rho$ is positive.
(From this, we will conclude that  $y=1/(x+\mu/2)$ gives a piece of the upper boundary of the intersection of the domains of the planar extensions.)
\end{rem}

\begin{lem}\label{lemAreaNatExt} The area of  
$\Omega_{r}$ is
\[ 
c_r = 2 \log\cot 
\frac{\pi}{2q}\;.
\]
\end{lem}

\begin{proof} From the discussion above, it suffices to find the Lebesgue area of $\Omega_{k}\,$, the planar extension of $k(x)\,$.    By symmetry, we easily find that this area is   
\[
\aligned 
c_r &= 2 \, \log \; \dfrac{2/\mu + 
Q\cdot1}{Q\cdot1-\mu/2}\,\prod_{j=1}^{n-2}\, 
\dfrac{Q\cdot\phi_{j+1}-Q\cdot(-\delta_j)}{Q\cdot\phi_j-Q\cdot(-\delta_j)}\, 
\dfrac{\vert \,Q\cdot\phi_j - Q\cdot1\,\vert }{\vert \,Q\cdot\phi_{j+1}- 
Q\cdot1\,\vert}
\\
          &= 2 \,  \log \; \dfrac{2/\mu + Q\cdot1}{Q\cdot 1-\mu/2}\,  
\dfrac{\vert \,Q\cdot\phi_1 - Q\cdot 1\,\vert }{\vert \,Q\cdot\phi_{n-1}- 
Q\cdot1\,\vert}\; \prod_{j=1}^{n-2}\, 
\dfrac{Q\cdot\phi_{j+1}-Q\cdot(-\delta_j)}{Q\cdot\phi_j-Q\cdot(-\delta_j)}\\
\\
 &= 2 \,  \log \; \dfrac{2/\mu + Q\cdot1}{1/(\sin \pi/q)}\; 
\prod_{j=1}^{n-2}\, \dfrac{\tan\frac{(j+1)\pi}{q}+U^{j}Q\cdot 1)}  
{\tan\frac{j \pi}{q}+U^{j}Q\cdot 1}\;.
\endaligned 
\]

But, $Q\cdot 1 = \mu/2 - 1/(\sin\pi/q)$, that is $Q \cdot 1 = (\cos \pi/q + 
1)/(\sin\pi/q)$. Its image under  $U^{j} = \begin{pmatrix} \cos j\pi/q & - 
\sin j\pi/q\\
                                  \sin j\pi/q   &  \cos 
j\pi/q\end{pmatrix}$ is thus
$\frac{\cos \frac{j \pi}{q} + \cos \frac{(j+1)\pi}{q}}{\sin \frac{j 
\pi}{q} + \sin \frac{(j+1)\pi}{q}}\, $.  
Let $W =  \prod_{j=1}^{n-2}\, 
\dfrac{\tan\frac{(j+1)\pi}{q}+U^{j}Q\cdot 1}  {\tan\frac{j 
\pi}{q}+R^{j/2}Q\cdot 1}$, then

\[
\aligned
W &=\dfrac{\cos \pi/q}{\cos \frac{(n-1)\pi}{q}}\,\prod_{j=1}^{n-2}\, 
\dfrac{\sin\frac{(j+1)\pi}{q}+\cos\frac{(j+1)\pi}{q}\, \frac{\cos 
\frac{j \pi}{q} + \cos \frac{(j+1)\pi}{q}}{\sin \frac{j \pi}{q} + 
\sin \frac{(j+1)\pi}{q}}}
{\sin\frac{j\pi}{q}+\cos\frac{j\pi}{q}\, \frac{\cos \frac{j \pi}{q} + 
\cos \frac{(j+1)\pi}{q}}{\sin \frac{j \pi}{q} + \sin 
\frac{(j+1)\pi}{q}}}\\
\\
 &= \dfrac{\cos \pi/q}{\sin \pi/q}\,
\prod_{j=1}^{n-2}\, \dfrac{ \sin\frac{(j+1)\pi}{q}\sin\frac{j\pi}{q} 
+1 + 
\cos\frac{(j+1)\pi}{q}\cos \frac{j \pi}{q} }
{1 +\sin\frac{j\pi}{q}\sin\frac{(j+1)\pi}{q}  + \cos \frac{j \pi}{q}\cos\frac{(j+1)\pi}{q}}\\
\\
&=  \dfrac{\cos \pi/q}{\sin \pi/q}\;.
\endaligned 
\]

Since $c_r =  2 \,  \log \; \dfrac{2/\mu + Q\cdot 1}{1/(\sin 
\pi/q)}\cdot W$, we have that $c_r =  2 \,  \log \; (2/\mu + Q\cdot 1)(\cos 
\pi/q)\,$.  This evaluates to  $2\, \log [\sin \pi/q + (\cos^2 
\pi/q+\cos \pi/q)/(\sin \pi/q)]\,$ that is, to  $2 \log\cot 
\frac{\pi}{2q}$.  One could express this as $c_r 
= 2 \log (Q\cdot1)$\,.
\end{proof}

\section{First return type}\label{s:firstReturn}

Natural extensions provide a basic tool for comparing continued fraction type interval maps.  We show that each of the interval maps of interest to us has a model of its natural extension given by the  first return map to a subset of the unit tangent bundle of a corresponding hyperbolic surface.      

\subsection{First return type defined} 

We give definitions allowing us to formalize the notion of first return type.  

\begin{defin}  For $M \in \slr$ and $x \in \mathbb R\,$ such that $M \cdot x \neq \infty\,$, let 
\[\tau(M,x) := -2 \,\log |c x + d\,|\,\,\]
where $(c,d)$ is the bottom row of  $M\,$ as usual.   
\end{defin}

As usual, we consider the projective group $\text{PSL}(2, \mathbb R)$ in lieu of $\text{SL}(2, \mathbb R)$.  
Elementary calculation shows that $\tau$ induces on $\text{PSL}(2, \mathbb R)$ a cocycle, in the following sense:
\[\tau(MN, x) = \tau(M, Nx) + \tau(N, x)\]
whenever all terms are defined and we choose each projective representative such that the corresponding $cx + d$ is positive.   (In all that follows,  the set where such a $c x +d$ is zero can be avoided.) 

\medskip 
    
\begin{defin}\label{d:returnTimeAndGamma}  Suppose that $f$ is a piecewise M\"obius interval map, say defined on an interval $I$, with     $I=\bigcup \, I_\alpha$ and for each $\alpha$, $f$ on  $I_{\alpha}$ given by $x \mapsto
M_\alpha  \cdot x\,$.  For each $x \in I$, the {\em return time} of $x$ is  $\tau_f(x) := \tau(M_\alpha\,,\, x)\,$.    Finally, let   $\Gamma_f$ be the group generated by the set of the $M_{\alpha}$.     
\end{defin}

In the following, we rely on terms and notation introduced in Section~\ref{ss:NatExtGeoFloIntro}, see especially  Equation~\eqref{eq:2Daction}.    Recall that a Fuchsian group is a discrete subgroup of $\text{SL}(2, \mathbb R)$  (or of $\text{PSL}(2, \mathbb R)$, depending upon context).

\begin{defin}\label{defReturnType} 
For $f$ as above, on $I \times \mathbb R$ we have the piecewise defined map $\mathcal T_f$ given by  taking each transformation $\mathcal T_{M_{\alpha}}$ above $I_{\alpha}$.   We say that $f$ {\em has a positive planar model} if there is a compact set $\Omega_f \subset I \times \mathbb R$ also fibering over $I$, and of positive Lebesgue measure, say $c_f$,  such that $\mathcal T_f( \Omega_f )= \Omega_f$.  One can prove (cf. \cite{AS3}) that such a set is unique, and   $\mathcal T_f$ is then a measurable automorphism of $\Omega_f$. We then refer to  the    marginal measure of $(1/c_f) \,dx \, dy$ (that is,  the measure on $I$ obtained by integrating along the fibres) simply as {\em the marginal measure}.

Further, we say that $f$ is of {\em first return type} if: 
\begin{enumerate}
\item   $f$ has a positive planar model;

\item $\Gamma_f$ is a Fuchsian group;

\item   for almost every  $x \in  I$ we have $\tau_f(x)> 0$;  and, 

\item  for almost every  $(x,y) \in \Omega_f$ and every non-trivial  $M \in \Gamma_f$ with $\mathcal T_M(x,y) \in \Omega_f$ and $\tau(M, x)\ge 0$, we have $\tau_f(x) \le  \tau(M, x)$.
\end{enumerate}
\end{defin}

\begin{rem}  Recall that,  with our usual notation,  the derivative of $M\cdot x$ at $x$ is $(cx +d)^{-2}\,$.      The positivity of the values $\log 1/(c x + d)$ shows that 
 any   (non-trivial) piecewise fractional linear map  of first return type  is expanding almost everywhere.   
\end{rem}  

\bigskip 
\subsection{Cross-section as natural extension} 
In tersest terms, a  measurable {\em cross-section} for the geodesic flow is a subset of the unit tangent bundle through which almost every geodesic passes tranversely and infinitely often, equipped with the transformation defined by first return under this geodesic flow.    Any flow invariant measure then induces an invariant measure on the cross-section and thus one can speak of the corresponding dynamical system.    Similarly, a {\em factor} of a dynamical system is a second dynamical system with a 
measurable map from the first to the second underlying spaces such that the corresponding diagrams of spaces and maps commute in the measurable category.       

Recall that a {\em natural extension} of a system  $(f, I, \mathscr B, \nu)$ is a dynamical system $(\mathcal T, \Sigma, \mathscr B', \mu)$ such that with the projection map $\pi: \Sigma\to I$ to the first coordinate,  four criteria are satisfied:   (1)  $\pi$ is a surjective and measurable map that pulls-back $\nu$ to $\mu$; (2)  $\pi \circ \mathcal{T} = f \circ \pi$;    (3) $\mathcal{T}$ is an invertible transformation; and,     (4) any invertible system that admits $(f, I, \mathscr B, \nu)$ as a factor must itself be a factor of $(\mathcal T, \Sigma, \mathscr B', \mu)$.    A standard method to verify this minimality criterion is to verify that $\mathscr{B}' = \bigvee_{n\ge 0}  \mathcal{T}^{n} \pi^{-1} \mathscr{B}$.   (In our setting,  the $\sigma$-algebras are always the appropriate Borel algebras.)    In the setting that $f$ is expanding,  it suffices to show that $\mathcal T^{-1}$ is expanding for $y$-values (that is, that  a.e. $(x,y) \in \Sigma$ has a neighborhood in which $\mathcal T^{-1}$ has this property);  see, say,  the proof of Theorem~1 of \cite{KSS} (on p. 2219 there).

It is well known that the geodesic flow on the unit tangent bundle of $\mathbb H$ is an Anosov flow  (indeed, sometimes called a  ``hyperbolic flow'') ---  there is a splitting into the direct sum of three invariant subbundles, with one tangent to the flow,  one expanded exponentially, one contracted exponentially; see  \cite{M}.

\begin{thm}\label{t:firstReturnTypeIsFirstReturn}  
If $f$ is of first return type, then there is a cross-section for the geodesic flow on the unit tangent bundle of $\Gamma_f\backslash \mathbb H$ such  that the first return map to this cross-section is a model of the  natural extension of the system defined by $f$ and the marginal measure.
\end{thm}
\begin{proof}  We first prove that the domain $\Omega_f$ projects to a cross-section for the geodesic flow on   $\Gamma_f\backslash\text{PSL}(2, \mathbb R)$, then we prove that the first-return map to this section is conjugate to $\mathcal T_f$. 

There is an injective map $\iota: \Omega_f \to \mathcal A \subset \text{PSL}(2, \mathbb R)$ given by sending 
$(x,y) \mapsto A = \begin{pmatrix} x&xy-1\\1&y\end{pmatrix}$.  (We use the standard abuse of using a matrix to represent its class in the projective group.)  

Just as the unit tangent bundle of $\mathbb H$ is given by $\text{PSL}(2, \mathbb R)$,  that of $\Gamma_f\backslash \mathbb H$ is given by $\Gamma_f\backslash\text{PSL}(2, \mathbb R)$.  Furthermore,  the geodesic flow on this quotient is simply of the (local) form $[A] \mapsto [A g_t]$,  where each $[B]$ here denotes the coset $\Gamma_f B$ and $g_t$ denotes the usual diagonal matrix.     

Let $\Sigma_f = \{ \, [A]\,|\,  A = \begin{pmatrix} x&xy-1\\1&y\end{pmatrix} \; \mbox{for}\;\; (x,y) \in \Omega_f\;\}\subset \Gamma_f\backslash\text{PSL}(2, \mathbb R)$.   We claim that the map $\Omega_f \to \Sigma_f$ given by sending $(x,y)$ to $[\iota(x,y)]$ is almost everywhere injective.   For this, let $A, A'$ be in $\iota(\Omega_f)$,  suppose that there is $M \in \Gamma_f$ such that $M = A' A^{-1}$.  Now,  
 \[ A' A^{-1} = \begin{pmatrix} 1 + x'(y-y')&(x'-x) - xx'(y-y')\\y-y'&1 - x(y-y')\end{pmatrix}\,.\]
If  $y=y'$ then  $M$ is a translation matrix $\begin{pmatrix} 1&x'-x\\0&1\end{pmatrix}$, such that $\mathcal T_M(x,y)=(x',y)$; but obviously, $\tau(M,x)=0$, and items three and four of the definition of the first-return type preclude the existence of such a matrix, except for a set of measure zero.

On the other hand, if $y-y' \neq 0$, then we first  find that $y-y'$, being the $(2,1)$-element of the matrix $M\in \Gamma_f$, belongs to a countable set; hence $x'= (a-1)/(y-y')$, where $a$ is the $(1,1)$-element of  $M$, belongs to a  countable set, as does $x$, and almost everywhere injectivity holds. 
 
Recall that  Liouville measure on  $T^1 \mathbb H$ is given  as the product of the hyperbolic area measure on $\mathbb H$ with the length measure on the circle of unit vectors at any point.  In \cite{AS2} (see especially Section~3 there), we show that  $\mathcal A$ gives a (measurable) transversal to the geodesic flow on  $T^1 \mathbb H$ in the sense that almost every geodesic meets $\mathcal A$ exactly once and that Liouville measure factors as $dx\,dy\,dt$ where $t$ is the variable for geodesic flow.   

It follows from the third and fourth items of the definition of first return type that the map $\Phi_f: \Sigma_f \to \Sigma_f$, given by $[A] \mapsto [M A g_{t_0}]$ with the various $M$ satisfying $M \cdot x = f(x)$ and the $t_0 = -2 \log (c x + d)$ correspondingly defined,   is in fact a {\em first} return map to the cross-section $ \Sigma_f $.

Since  $A\cdot i = (x i + xy-1)(i + y)$,  one finds that $\mathcal A_x = \{ A\in \mathcal A\,|\, x \;\mbox{is fixed}\}$ corresponds to the horocycle of $\mathbb H$ of Euclidean radius $1/2$ based at $x$.   Therefore,  the fiber of $\Sigma_f$ above any $x \in I$ lies on a  horocycle, and in particular lies in the strong unstable manifold of the geodesic flow, viewed as an Anosov flow.    In other words,  the geodesic flow is contracting in this ``$y$''-direction.   Since $\Phi_f$ is a return map of this flow, it follows that   $\Phi_{f}^{-1}$ acts as an expanding map on the $y$-values.    For the $x$-values,   $\Phi_f$ is expanding, as it agrees with $f$ on these.  
\end{proof}

\begin{cor}\label{c:firstReturnIsErgodic}  
If $f$ is of first return type and $\Gamma_f$ is of finite covolume, then $f$ is ergodic with respect to the marginal measure.
\end{cor}
\begin{proof} By the well-known results of Hopf,  when $\Gamma_f$ is of finite covolume the geodesic flow on unit tangent bundle of $\Gamma_f\backslash \mathbb H$ is ergodic.    But, this ergodicity then implies that of the induced map,  $\phi: \Sigma_f \to \Sigma_f$.     Finally, the ergodicity is a property shared by a transformation and its natural extension.
\end{proof}

 \medskip 
Not every $f$ with a planar model of a natural extension is of first return type, as we showed in \cite{AS2}.   Rather, the entropy of $f$ must accord with the measure of $\Omega_f$,  as the following shows.

\begin{prop}\label{p:entropyFirstReturn}  
Suppose that $f$ is a piecewise M\"obius interval map with a positive planar model of its natural extension, that 
the group $\Gamma_f$ is Fuchsian of finite covolume, and that $f$ is ergodic with respect to the marginal measure.  Then $f$ is of first return type if and only if the volume of the unit tangent bundle of the surface uniformized by $\Gamma_f$ equals the product of the entropy of $f$ with the Lebesgue measure of $\Sigma_f$.
\end{prop}
\begin{proof} For ease of typography, let $\Sigma= \Sigma_f$.

By Rohlin's formula for the entropy of an ergodic interval map,  see say \cite{DK},
and the fact that locally $f(x) = (a x + b)/(c x + d)$, we have

\[
\begin{aligned}
h(T) &= \int_I \log |f'(x)|\, d\nu \\
        &= \int_I -2 \log |c x + d| \,d \nu\\
        &= \int_I \tau_f(x) \,d \nu\\ 
         & = \dfrac{1}{c_f} \,  \int_{\Sigma} \tau_f(x) \,dx\,dy\\
         & = \dfrac{ \int_{\Sigma} \tau_f(x)\, dx \, dy}{ \lambda(\Sigma)}\\
         &\ge  \dfrac{\text{vol}(T^1(\Gamma\backslash \mathbb H))}{ \lambda(\Sigma)}\,.
\end{aligned} 
\]

\noindent
The final inequality holds since the {\em first} return map to $\Sigma$ is given by following the unit tangent vectors along the geodesic arcs with unit tangent vectors in $\Sigma$ until a first return to $\Sigma$. The fact that $\mathcal A$ is a transversal implies that these geodesic arcs sweep out the unit tangent bundle (up to measure zero).     Hence,  equality holds if and only if $f$ is of first return type.        
\end{proof} 

\begin{rem}   In the case of our previous paper \cite{AS2},  some of the interval maps treated were piecewise given by matrices of determinant $-1$, causing the cross-section to consist of two copies of the natural domain of the underlying interval map,  and hence there is then a factor of 1/2 in the above.   We have simplified the argument given there for the analogous result --- there we relied on a result of Sullivan stating that the entropy of the geodesic flow on the unit tangent bundle of the uniformized surface (with respect to normalized Liouville measure) equals one.   Here, we see that the existence of a first return time interval map $f$, combined with Abramov's formula, shows that the entropy of this flow on the unit tangent bundle of $\Gamma_f\backslash \mathbb H$ must equal one. 
\end{rem}

\subsection{Our continued fraction maps are of first return type}\label{ss:oursFirst}

\begin{prop}\label{p: RosenSymmIsFirst}   
The symmetric Rosen fraction map  $h: J \to J$, as defined in Equation \eqref{e:rosenSymDef},   is of first return type.  
\end{prop}  

\begin{proof} Nakada \cite{N} has shown that this map has the correct entropy value so that we can apply Proposition~\ref{p:entropyFirstReturn}.
\end{proof}

Our use of the geodesic flow on  $\mathcal A \subset \slr$ rather quickly identifies the region $\Omega_f\,$, as opposed to the often quite tedious approach of directly  building a cross-section for which the first return map induces the function, see  \cite{Se},  \cite{AF}, \cite{MS}.    On the other hand,  in our approach one is left to prove that it is the {\em first} return map that induces $f$.   One can do this by evaluating integrals  as in Proposition ~\ref{p:entropyFirstReturn}. However, this can be quite challenging, as these integrals naturally involve polylogarithmic functions.  Indeed, in general the existence of an interval map of first return type implies an identity of polylogarithmic functions, see \cite{F}.      Here we use specific aspects of maps to show that they must indeed be of first return type. \\

\begin{prop}\label{propVeechAddIsFirst}   Suppose that $q = 2n$ and let $\mu = \mu_q = 2 \cot \pi/q\,$.   Let $\mathbb I = (-\mu/2, \mu/2)\,$,   and $P, R$ be the elements of $\text{PSL}(2, \mathbb R)$ representing translation by $\mu$ and the rotation standard rotation by $2\pi/q\,$.     Then $V_q\,$, the group generated by $P, R\,$,  is the (internal) free product of the cyclic groups generated by these two elements; in particular, this is isomorphic to $\mathbb Z \star \mathbb Z/n \mathbb Z\,$.    

For $j \in \{1, n-1\}$, let  $d_j = R^{-j}\cdot \mu/2\,$.  Each interval $ (d_j, d_{j+1}\,)$ is contained in $\mathbb I\,$.   The {\em additive Veech interval map} is 
\[
\begin{aligned} 
f : \mathbb I &\to \mathbb I\\
      x &\mapsto P^k R^j\cdot x\,,
\end{aligned}
\] 
whenever 
$x \in (d_j, d_{j+1}\,)$, where $k \in \mathbb Z$ is the unique power of $P$ such that the image lies in $\mathbb I\,$.   

The map $f$ is of first return type.  
\end{prop}  

\begin{proof}  Recall that $\Delta(k, j)$ denotes the {\em cylinder} of  $P^k R^j$.   
It is easily checked that each $x \mapsto P^k R^j \cdot x$ is expanding  almost everywhere on $\Delta(k, j)\,$.  
Thus, $\tau(f,x) \ge 0$ for all $x \in \mathbb I\,$.

  Recall that the action of $M$ sends the curve  of equation $y = 1/(x - \delta)$   to that of the equation $y = 1/(x - M\cdot \delta)$,  and $(x,y) \in \Omega_{a}$ implies that  $\delta \in \mathbb R \setminus \mathbb I$ (or $\delta = \pm \mu/2$).

By the structure of $V_q\,$ we can uniquely express $M$ in the following form 
\[ M = P^{t_m} R^{s_m} \cdots P^{t_1}R^{s_1}\,\]
with $s_i \in \{1, \dots, n-1\}$ if $i>1\,$; $s_1 \in \{0, \dots, n-1\}\,$; $t_i \in \mathbb Z \setminus \{0\}$ if $i<m$; and, $t_m \in \mathbb Z\,$.

We thus fix $(k,j)$, a subinterval $J \subset \Delta(k,j)$, and $M$ in the above form that sends $J$ into $\mathbb I\,$.      We let $\mathcal J \subset \Omega_{a}$ denote the subset fibering over $J$.

\bigskip

\noindent
{\bf Case:  $s_1 \notin \{0, j\}$.}     Since $s_1 \neq j$, we have that for any $x \in J$,  $R^{s_1}\cdot x \in \mathbb I$.   On the other hand,  $R^{s_1}$ maps all $\delta$ in the complement of $\mathbb I$ into $\mathbb I$.   That is,  each   $(x,y) \in \mathcal J$ has image under $R^{s_1}$ that lies outside of $\Omega_{a}$ due to the value of its $y$-coordinate.   Clearly, $M \neq R^{s_1}$ and hence  $t_1 \neq 0$;   but   each   $(x,y) \in \mathcal J$ has image under $P^{t_1} R^{s_1}$ that lies outside of $\Omega_{a}$ now due to the value of its $x$-coordinate.    Since the expansion of $M$ as a word in $P$ and positive powers of $R$ has finite length,  these alternating types of obstruction persist, and we conclude that there is no $M$ with $s_1 \notin \{0, j\}$ that can send all of $\mathcal J$ into $\Omega_{a}$.  

\bigskip

\noindent
{\bf Case:  $s_1 =  0$.}    Of course here $t_1 \neq 0$, but then $P^{t_1}$ sends $J$ outside of $\mathbb I$ and maps a subset  of the 
$\delta \in \mathbb R \setminus \mathbb I$ back into $ \mathbb R \setminus \mathbb I$.   For these values,   we  then proceed by invoking the alternating obstruction argument of the previous case.   For those $\delta \in \mathbb R \setminus \mathbb I$  that $P^{t_1}$ maps into $\mathbb I$,  
there could be some $R^{s_2}P^{t_1}$ mapping the corresponding set of $(x,y)$ into $\Omega_a$;  however,  this map is the inverse of $P^{-t_1}R^{-s_2}$, and under the assumption that the return time under $M$ is positive, the return time of  $MP^{-t_1}R^{-s_2}$ is  less than that of $P^{-t_1}R^{-s_2}$.  That is, this second possibility reduces to replacing $x$ by $R^{s_2}P^{t_1}\cdot x$.  
\bigskip

\noindent
{\bf Case:  $s_1 =  j$.}    In this case, we again cannot have $t_1 = 0$.     If  nonzero $t_1 \neq k$,  then an alternating failure of new $x$- or $y$-coordinate again appears.   Thus,  we can only have that $t_1 = k$.   It is then trivial (by the cocycle property of return times) that $M$ has at least the return time of $P^{k}R^j$.   
 
\end{proof}   

\begin{defin}  Let $f$ be a piecewise fractional linear map on some interval $I\,$.  With notation as above, suppose that for some index value $\alpha$ one has $M_{\alpha}$ parabolic and that the fixed point of $M_{\alpha}$ lies in $I_{\alpha}\,$.      Then the {\em parabolic acceleration} of $f$ with respect to $M_{\alpha}$ is the piecewise fractional linear map on $I$ that agrees with $f$ except on $I_{\alpha}$, and such that for $x \in I_{\alpha}$ this new map is given by applying the appropriate power of $M_{\alpha}$ to escape $I_{\alpha}\,$.\end{defin}

The following is easily proved.
\begin{lem}  Suppose that $f$ is a piecewise fractional linear map on some interval $I\,$.    Suppose that   $g$ is a parabolic acceleration of $f$ that generates the same Fuchsian group.  Then $g$ is of first return type if and only $f$ is.
\end{lem}

\begin{cor}  The multiplicative Veech map is of first return type.
\end{cor}

\begin{proof} 
By construction,  the multiplicative Veech map is the (two-fold)
parabolic acceleration of the Veech additive map in the indifferent fixed points.   Therefore, the result follows from the previous lemma and Proposition ~\ref{propVeechAddIsFirst}.
\end{proof} 
 
\section{The comparison}\label{s:comparison}

\subsection{The intersection, $\overline{\Omega} := \Omega_v\,\cap 
\,\Omega_r$}

\begin{prop} Let $q\ge 8$ be an even natural number, and let  $\Omega_r$ and
$\Omega_v$ be the domains of the natural extensions of 
$r(x)$ and $v(x)$, respectively.  Let     
$\overline{\Omega} = \Omega_r \cap \Omega_v$.   The upper boundary of $\overline{\Omega}$  is given by : 

\[
\begin{cases} 
y = \frac{1}{x+M \cdot 1} & \text{if} \;\; x\in[-\mu/2,\;2/\mu)\;;\\
\\
y = \frac{1}{x+\mu/2} & \text{if} \;\; x\in[\;2/\mu,\;\mu/2)\;;\\
\end{cases}
\]
and the lower boundary of $\overline{\Omega}$ by : 
\[
\begin{cases} 

y = \frac{1}{x-\mu/2}& \text{if} \;\; x\in[-\mu/2,-2/\mu)\;;\\
\\
y = \frac{1}{x-M \cdot 1}  & \text{if} \;\; x\in[-2/\mu,\;\mu/2)\;.
\end{cases}
\]

The Lebesgue area of $\overline{\Omega}$  is $\lambda(\overline{\Omega}) = 2 \, \log 
\left(\, \cos \frac{\pi}{q}(1+\cos \frac{\pi}{q})\right)$.

\end{prop}  

\begin{proof}  The boundaries of $\Omega_v$ are given in Proposition~\ref{p:VeechMultNat}, those of $\Omega_r$ in Lemma~\ref{lemNatExtDouble}.  In particular, the upper boundary of  $\Omega_v$ above the left end of $I$ is  of equation $y = 1/(x+\gamma)$, whereas the leftmost upper boundary of $\Omega_r$ is given by $y = 1/(x+ M . 1)\,$.   Recall that 
 $M . 1 = \mu/2 + 1/(\sin \pi/q)$ whereas $\gamma = \mu/2 \cdot 
(5\mu^2 +4)/(3 \mu^2 -4)$.   As  $1/(\sin \pi/q) = \sqrt{4+\mu^2}/2$ and $\gamma = \mu/2 + 
\mu(\mu^2+4)/(3 \mu^2 - 4)$,    elementary calculations show that as soon as $q \ge 8\,$, one has  $M . 1 > 
\gamma\,$.     Now, it is easily checked that (for all of our $q$),  $y = 1/(x+\gamma)$ does give the lowest piece of the upper boundary of $\Omega_v\,$.    Thus,   we conclude  that $y = 1/(x+ M . 1)\,$ gives the upper boundary of $\overline{\Omega}$ above all of $[-\mu/2,\;2/\mu)\,$.

For $x>2/\mu$, the lowest piece of the upper boundary of  $\Omega_r$ is of equation  $y=1/(x- M \cdot( -\delta_1)\;)\,$.     Using Lemma~\ref{lemUsefulIds}, one finds that  
$-M \cdot ( -\delta_1)$ is smaller than $\mu/2\,$.     Thus, one finds that the second piece of the boundary of $\overline{\Omega}$ is as claimed.   Symmetry completes the description of this boundary. 

The area is calculated as in Proposition~\ref{p:VeechMultNat} and Lemma~\ref{lemNatExtDouble}. 
\end{proof}

\begin{figure}
\begin{center}
\includegraphics{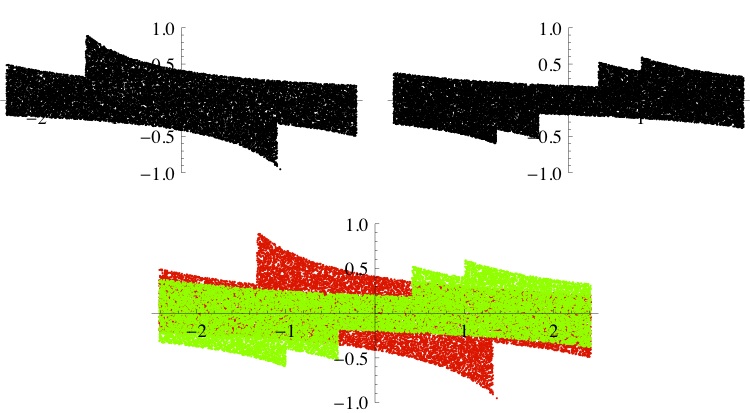}
\caption{Overlay of the domains $\Omega_v$ and $\Omega_r\,$, $q=8\,$.}
\label{natExtOverlayFig}
\end{center}
\end{figure}
 
Were the two algorithms the same,  we would of course find that the quotient of the area of their respective domains of planar natural extensions  by the area of the intersection would equal one.  Rather, in the limit we find that the intersection accounts for $1/3$ of the area of the natural extension domain for the multiplicative Veech map,  but for a negligible amount of that of the doubled symmetric Rosen map.
\begin{cor}  As $q$ goes to infinity the limit  of the ratio of the area of the intersection to the original domain of natural extensions is: 
\[ \lim_{q \to \infty} \dfrac{\lambda(\, \overline{\Omega}\,)}{ \lambda(\, \Omega_v\,)} = 1/3 
\;\;\; \mbox{and}\;\;\;\lim_{q \to \infty} \dfrac{\lambda(\, \overline{\Omega}\,)}{ \lambda(\, \Omega_v\,)} = 0\,.\]
\end{cor}

\subsection{Unbounded return times}\label{ss:unbounded}

Since both the Rosen and the Veech maps are of first return type, the return maps to the intersection $\overline{\Omega}$
 by iterations of  ${\mathcal T}_r$ and  ${\mathcal T}_v$, respectively,  actually define the same map on this intersection. 
 We now briefly consider the number of iterations required for these return maps.

\begin{defin} Let  $F: \Omega\to\Omega$ and $\Lambda \subset 
\Omega$.  For $p \in \Omega$, the {\it induction index } with respect to 
 $\Lambda$ of  $F$ at $p$ is the smallest positive exponent $n$ such that 
$F^n(p) \in \Lambda$.   
\end{defin}

We now determine the region where where both maps have induction index equal to one.

\begin{lem} For $(x,y) \in \overline{\Omega}$, the equality

\[{\mathcal T}_r(x,y)={\mathcal T}_v(x,y)\]

if and only if either 

\begin{enumerate}
\item 
\[  \pm x \in   [R \cdot \mu/2, \beta)\,;\]
or,
\item  there is an  $n\in {\mathbb N}$ such that 
\[\pm x \in  (RP^{-1})^n(\; [-\mu/2, R \cdot 0)\;) \;.\]
\end{enumerate}

In particular, the induction index of either of the functions equals one at any of these points.
\end{lem}

\begin{proof} It is clear that the subregions where the maps agree are defined simply in terms of $x$.
Furthermore,  the only possibility for such an equality is where both 
$v(x)$ and $r(x)$ are given by the same element of $\text{PSL}(2, {\mathbb R})$.   By symmetry, it suffices to check for such equality when $x >0$.

For  $x \in [0, R \cdot \mu/2)$, $r(x) = R^{-1} \cdot x$ whereas $v(x) = 
P^k R^{-n} \cdot x$  with $n\ge 2$.  That is, there is no equality here.

For $x \in [R \cdot \mu/2, \beta[$, both $r(x)$ and $v(x)$  are equal to
$P^kR^{-1} \cdot x$  and furthermore, one verifies that the exponents $k$ are the same for these two maps. 

If $x\in [\beta, \mu/2)$, then $v(x) = (PR^{-1})^k \cdot x$ for an appropriate value of $k$. 
However,  $r(x)$ is only of this form on (a subinterval of each of) its non-full cylinders --- thus, where 
$x \in  
(RP^{-1})^n(\; [-\mu/2, R \cdot 0)\;)$. 

 Since the two functions take the same values at these point,  we do indeed find that both have induction index there. 
\end{proof}

 In Lemma~\ref{l:slowReturn} we will show that the induction  index of $v(x)$ is unbounded.   The following is helpful in the proof of that result. 

\begin{lem}\label{l:anId} The following equality holds for $Q$ the conjugation matrix of Equation~\eqref{eq:conjMatM}.
$$\mu+\frac{1}{Q \cdot 1} = Q\cdot 1\;.$$
\end{lem}

\begin{proof}  By expressing all values in terms of elementary trigonometric functions, the result follows immediately.
\end{proof}

\begin{lem}\label{l:slowReturn}  Suppose that $q$ is divisible by $4$. 
Then for any $(x,y) \in \overline{\Omega}$ such that  $x \in 
(R^{q/4}P^{-1})^k(\;[-\mu/2, \mu/2)\;)$, 
the induction index of ${\mathcal T}_v$ at $(x,y)$ is at least $k$.  
\end{lem}
  
\begin{proof}   The multiplicative Veech interval map is such that the subinterval upon which $x \mapsto P R^{q/4} \cdot x$ defines a full cylinder.   Thus, for each $k \in \mathbb N$,  there is a non-empty subinterval for which the $v$-orbit of $x$ begins with $x$ and then $( P R^{q/4} \cdot x, \dots, (P R^{q/4})^k \cdot x)$.  (Of course, these  $x$ are exactly those of our hypothesis: $x \in (R^{q/4}P^{-1})^k(\;[-\mu/2, \mu/2)\;)$.)  We show that the strip lying above this interval has image under ${\mathcal T}_v, {\mathcal T}^{2}_v, \dots, {\mathcal T}^{k-1}_v$ exterior to the intersection domain  $\overline{\Omega}$.

When $q$ is divisible by $4$,  as a projective element, $R^{q/4}$ has order $2$.  In fact,  $R^{q/4}\cdot  x = -1/x$.    By Lemma~\ref{l:anId},   $PR^{q/4}  \cdot (-Q \cdot 1) = \mu +1/(Q \cdot 1) = Q \cdot 1$.  Recall that $R^{q/4}P^{-1}(\;[-\mu/2, 
\mu/2)\;)$ is a subinterval of  $[0,R \cdot \mu/2)$.   Therefore the intersection of the vertical fibres with $\overline{\Omega}$ lie between the curves of equation   $y = 1/(x \pm Q \cdot 1)$.   By Equation~\eqref{eq:TMonDelta},  $PR^{q/4}$ sends the upper boundary curve of equation
 $y = 1/(x + Q \cdot 1)$ to the lower boundary curve,  of equation  $y = 1/(x - Q \cdot 1\;)$.

 We now consider further the orbit of $\delta = - (Q\cdot 1)$ under powers of  $PR^{q/4}$.   First note that $PR^{q/4}$  sends $0$ to $\infty$ and fixes $ \frac{\mu\pm \sqrt{\mu^2-4}}{2}$.  Since $PR^{q/4}$ sends $\delta = -Q \cdot 1$ to  the value $Q \cdot 1 = \frac{\mu+\sqrt{\mu^2+4}}{2}$ which is greater than the larger of the fixed points of this hyperbolic matrix,   the values of   $(PR^{q/4})^{1+j}\cdot (-Q\cdot 1)$ decrease with positive $j$, but remain greater than the fixed point value $ \frac{\mu + \sqrt{\mu^2-4}}{2}$.   
 
Again applying Equation~\eqref{eq:TMonDelta}, we find that the region of $\overline{\Omega}$ fibering over the interval   $(R^{q/4}P^{-1})^k(\;[-\mu/2, \mu/2)\;)$ has image under the $\mathcal T_{P(R^{q/4})^j}$ with $1 \le j < k$ lying exterior to $\overline{\Omega}$.   It then follows that  the induction index of ${\mathcal T}_v$ at any $(x,y)$ with  $x \in 
(R^{q/4}P^{-1})^k(\;[-\mu/2, \mu/2)\;)$ is at least $k$.   
\end{proof}


\end{document}